\documentclass[12pt,english, reqno]{amsart}

\usepackage{fullpage}
\usepackage{hyperref}
\usepackage{amssymb}
\newcommand{\Qq}{\mathbb{Q}} 
\newcommand{\Rr}{\mathbb{R}} 

\numberwithin{equation}{section}

{\theoremstyle{plain}
\newtheorem{theorem}{Theorem}[section]
\newtheorem{lemma}[theorem]{Lemma}
\newtheorem{proposition}[theorem]{Proposition}
\newtheorem{conjecture}[theorem]{Conjecture}
\newtheorem{corollary}[theorem]{Corollary}}

{\theoremstyle{remark}
\newtheorem{remark}[theorem]{Remark}
\newtheorem{definition}[theorem]{Definition}}

\title{On finite embedding problems with abelian kernels}

\author{Fran\c{c}ois Legrand}

\email{francois.legrand@unicaen.fr}

\address{Normandie Univ., UNICAEN, CNRS, Laboratoire de Math\'ematiques Nicolas Oresme, 14000 Caen, France}

\begin{document}

\maketitle

\vspace{-7.5mm}

\begin{abstract}
Given a Hilbertian field $k$ and a finite set $\mathcal{S}$ of Krull valuations of $k$, we show that every finite split embedding problem $G \rightarrow {\rm{Gal}}(L/k)$ over $k$ with abelian kernel has a solu\-tion ${\rm{Gal}}(F/k) \rightarrow G$ such that every $v \in \mathcal{S}$ is totally split in $F/L$. Two applications are then given. Firstly, we solve a non-constant variant of the Beckmann--Black problem for solvable groups: given a field $k$ and a non-trivial finite solvable group $G$, every Galois field extension $F/k$ of group $G$ is shown to occur as the specialization at some $t_0 \in k$ of some Galois field extension $E/k(T)$ of group $G$ with $E \not \subseteq \overline{k}(T)$. Secondly, we contribute to inverse Galois theory over division rings, by showing that, for every division ring $H$ and every automorphism $\sigma$ of $H$ of finite order, all finite semiabelian groups occur as Galois groups over the skew field of fractions $H(T, \sigma)$ of the twisted polynomial ring $H[T, \sigma]$.
\end{abstract}

\section{Introduction} \label{sec:intro}

\subsection{Main result} \label{ssec:intro_1}

The inverse Galois problem over a field $k$, a question going back to Hilbert and Noether, asks whe\-ther every finite group occurs as a Galois group over $k$, i.e., as the Galois group of a Galois field extension of $k$.  A stronger version asks for solutions to finite embedding problems over $k$. As in, e.g., \cite[\S16.4]{FJ08}, say that a {\it{finite embedding problem}} over $k$ is an epimorphism $\alpha : G \rightarrow {\rm{Gal}}(L/k)$, where $G$ is a finite group and $L/k$ a Galois field extension, and that $\alpha$ {\it{splits}} if there is an embedding $\alpha' : {\rm{Gal}}(L/k) \rightarrow G$ such that $\alpha \circ \alpha' = {\rm{id}}_{{\rm{Gal}}(L/k)}$. A {\it{solution}} to $\alpha$ is an isomorphism $\beta : {\rm{Gal}}(F/k) \rightarrow G$, where $F$ is a Galois field extension of $k$ containing $L$, such that $\alpha \circ \beta$ is the restriction map ${\rm{Gal}}(F/k) \rightarrow {\rm{Gal}}(L/k)$.

The following conjecture is a special case of the D\`ebes--Deschamps conjecture (see \cite{DD97b}):

\begin{conjecture} \label{conj:DD}
Each finite split embedding problem over any Hilbertian field has a solution.
\end{conjecture}

\noindent
Recall that a field $k$ is {\it{Hilbertian}} if Hilbert's irreducibility theorem holds over $k$. For example, all global fields are Hilbertian. We refer to, e.g., \cite{FJ08} for more on Hilbertian fields.

The main interest of Conjecture \ref{conj:DD} is that it generalizes and unifies se\-ve\-ral conjectures in classical inverse Galois theory. For example, it allows to solve both the inverse Galois problem over every Hilbertian field (in particular, over $\Qq$) and the Shafarevich conjecture, the latter asserting that the absolute Galois group of the maximal cyclotomic extension of $\Qq$ is profinite free and being intensively studied (see, e.g., \cite{Pop96, HS05, Par09, Des15}). 

So far, Conjecture \ref{conj:DD} has been proved if the base field is a rational function field of one variable over an ample field\footnote{Recall that a field $k$ is {\it{ample}} (or {\it{large}}) if every smooth geometrically irreducible $k$-curve has zero or infinitely many $k$-rational points. Ample fields include algebraically closed fields, the complete valued fields $\Qq_p$, $\mathbb{R}$, $\kappa((Y))$, the field $\Qq^{\rm{tr}}$ of totally real numbers, etc. See, e.g., \cite{Jar11, BSF13, Pop14} for more details.} or if the base field is ample and Hilbertian (see \cite{Pop96, HJ98b}). On the other hand, no counter-example to Conjecture \ref{conj:DD} is known.

There is, however, a wea\-ker result, originating in a work of Ikeda (see \cite{Ike60}), which holds over every Hilbertian field (see, e.g., \cite[Proposition 16.4.5]{FJ08}): 

\vspace{1mm}

\noindent
($*$) {\it{Given a Hilbertian field $k$, every finite split embedding problem $\alpha : G \rightarrow {\rm{Gal}}(L/k)$ over $k$ with abelian kernel has a solution $\beta :{\rm{Gal}}(F/k) \rightarrow G$.}}

\vspace{1mm}

Our main aim is to show that there are solutions $\beta$ to $\alpha$ which are ``totally split" at any finitely many given Krull valuations of $k$. In the next result, which improves on Statement ($*$), we let $k_v^h$ denote the Henselization of a field $k$ at a Krull valuation $v$ of $k$ (see \S\ref{ssec:bas_3}).

\begin{theorem} \label{thm:intro_1}
Let $k$ be a Hilbertian field, $\mathcal{S}$ a finite set of Krull valuations of $k$, and $\alpha : G \rightarrow {\rm{Gal}}(L/k)$ a finite embedding problem over $k$. Assume $\alpha$ splits and ${\rm{ker}}(\alpha)$ is abelian. Then $\alpha$ has a solution ${\rm{Gal}}(F/k) \rightarrow G$ such that $F \subseteq L\cdot k_v^h$ for every $v \in \mathcal{S}$.
\end{theorem}

We point out that there are other results on solving finite embedding problems over Hilbertian fields with local conditions (see, e.g., \cite{KM04, JR18, JR19, JR20, FL20, CF21}), possibly with a wider range of local conditions and/or kernels. However, those results hold only over some specific classes of Hilbertian fields and consider more restrictive classes of valuations.

The proof of Theorem \ref{thm:intro_1} (see \S\ref{sec:proof_1}) follows the main lines of that of ($*$). Recall that the latter mainly consists in viewing $G$ as a quotient of the wreath product ${\rm{ker}}(\alpha) \wr {\rm{Gal}}(L/k)$, in realizing the latter as the Galois group of a Galois field extension of $k(T_1, \dots, T_n)$ containing $L$, where $n = [L:k]$, and in using the Hilbertianity of $k$ to specialize $T_1, \dots, T_n$ suitably. To get our extra local conclusion, we use various tools, such as ramification in higher dimension, several lemmas of general interest on specializations of function field extensions in several variables that we establish, and the fact that, given a field $k$, every finite abelian group is the Galois group of a Galois extension $E$ of the rational function field $k(T)$ with $E \subseteq k((T))$.

\subsection{Applications} \label{ssec:intro_2}

We then give two applications of Theorem \ref{thm:intro_1}, to the Beckmann--Black problem (see \S\ref{sssec:intro_2.1}) and inverse Galois theory over division rings (see \S\ref{sssec:intro_2.2}).

\subsubsection{On the Beckmann--Black problem} \label{sssec:intro_2.1}

The geometric approach to realize a finite group $G$ as a Galois group over a number field $k$ consists in realizing $G$ as a regular Galois group over $k$, i.e., as the Galois group of a Galois field extension $E/k(T)$ with $E \cap \overline{k}=k$. For such an extension $E/k(T)$, Hilbert's irreducibility theorem provides infinitely many $t_0 \in k$ such that the {\it{specialization}} $E_{t_0}/k$ of $E/k(T)$ at $t_0$ has Galois group $G$ (see \S\ref{ssec:bas_1} for basic terminology). Several finite groups, including non-abelian simple ones, have been realized by this method. See \cite{Ser92, Vol96, FJ08, MM18} for more details and references within.

The Beckmann--Black problem asks whether this approach is optimal: given a field $k$ and a finite group $G$, is every Galois field extension $F/k$ of group $G$ the specialization at some $t_0 \in k$ of some Galois extension $E/k(T)$ of group $G$ with $E \cap \overline{k} = k$? If $k$ is ample, the answer is {\it{Yes}} for every finite group $G$ (see \cite{CT00, MB01}). If $k$ is a number field, the answer is known to be {\it{Yes}} for only a few groups $G$, including abelian groups (see \cite{Bec94, Deb99a}), symmetric groups, and alternating groups (see \cite{Mes90, KM01}), and no counter-example is known. See also th\'eor\`eme 2.2 of the survey paper \cite{Deb01b}.

Given a number field $k$ and a finite solvable group $G$, it is widely unknown whether $G$ is a regular Galois group over $k$. As $G$ is a Galois group over $k$ by Shafarevich's theorem (see \cite[Theorem 9.6.1]{NSW08}), a positive answer to the Beckmann--Black problem for $G$ over $k$ is even harder to reach. The next theorem (see Corollary \ref{thm0}), which was not available in the literature, establishes a necessary condition for such an affirmative answer:

\begin{theorem} \label{thm:intro_3}
Let $G$ be a non-trivial finite solvable group, $k$ an arbitrary field, and $F/k$ a Galois field extension of group $G$. There exist $t_0 \in k$ and a Galois field extension $E/k(T)$ of group $G$ with $E \not \subseteq \overline{k}(T)$ such that the specialization $E_{t_0}/k$ of $E/k(T)$ at $t_0$ equals $F/k$.
\end{theorem}

In fact, we have the following stronger result. Given a finite embedding problem $\alpha : G \rightarrow {\rm{Gal}}(L/k)$ over a field $k$, denote the restriction map ${\rm{Gal}}(L(T)/k(T)) \rightarrow {\rm{Gal}}(L/k)$ by ${\rm{res}}$, and say that a solution to the finite embedding problem ${\rm{res}}^{-1} \circ \alpha$ over $k(T)$ is a {\it{geometric solution}} to $\alpha$. We derive from Theorem \ref{thm:intro_1} that, if ${\rm{ker}}(\alpha)$ is abelian, then every solution ${\rm{Gal}}(F/k) \rightarrow G$ to $\alpha$ occurs as the specialization at some $t_0 \in k$ of some geometric solution ${\rm{Gal}}(E/k(T)) \rightarrow G$ to $\alpha$ with $E \cap \overline{k}=L$ (see \S\ref{ssec:bas_2} for more details on specializations at the level of finite embedding problems and Theorem \ref{thm:intro_2} for our precise result). 

That is, our result solves the Beckmann--Black problem for finite embedding problems with abelian kernels over arbitrary fields and, in particular, widely extends both D\`ebes' theorem (see \cite{Deb99a}) solving the Beckmann--Black problem for abelian groups over arbitrary fields (the case $L=k$ of our result) and the affirmative answer to the Beckmann--Black problem for {\it{Brauer embedding problems}} (see, e.g., \cite[Chapter IV, \S7.1]{MM18} for more details). To our knowledge, only one other result about the Beckmann--Black problem for finite embedding problems was available in the literature. Namely, the answer is affirmative if the base field is PAC\footnote{Recall that a field $k$ is {\it{Pseudo Algebraically Closed}} (PAC) if every non-empty geometrically irreducible $k$-variety has a Zariski-dense set of $k$-rational points. See, e.g., \cite{FJ08} for more on PAC fields.}, as a combination of a result of Bary-Soroker (see \cite[Corollary 3.4]{BS09b}) and the result asserting that every finite embedding problem $G \rightarrow {\rm{Gal}}(L/k)$ over a PAC field $k$ has a geometric solution ${\rm{Gal}}(E/k(T)) \rightarrow G$ with $E \cap \overline{k}=L$ (see \cite{Pop96}).

\subsubsection{Inverse Galois theory over division rings} \label{sssec:intro_2.2}

According to Artin, if $H \subseteq L$ are division rings\footnote{A {\it{division ring}} is a non-zero ring $H$ such that every non-zero element of $H$ is invertible in $H$. Of course, commutative division rings are nothing but {\it{fields}}.}, $L/H$ is {\it{Galois}} if every element of $L$ which is fixed under all automorphisms of $L$ fixing $H$ pointwise is in $H$ (see \cite{Jac64, Coh95} for more on Galois theory of division rings). With the latter definition, inverse Galois theory can be studied over arbitrary division rings. Yet, it is striking that
not much is known in the non-commutative case (see \cite{DL20, ALP20, Beh21, BDL20f, Des21a, FL20} for recent works on that topic).

Given an automorphism $\sigma$ of a division ring $H$, let $H[T, \sigma]$ be the ring of polynomials $a_0 + a_1 T + \cdots + a_n T^n$ with $n \geq 0$ and $a_0, \dots, a_n \in H$, whose addition is defined componentwise and multiplication fulfills $Ta = \sigma(a) T$ for $a \in H$. By $H(T, \sigma)$, we mean the unique division ring containing $H[T, \sigma]$ and each element of which can be written as $ab^{-1}$ with $a \in H[T, \sigma]$, $b \in H[T, \sigma] \setminus \{0\}$.  If $H$ is a field and $\sigma= {\rm{id}}_H$, we retrieve the usual commutative polynomial ring $H[T]$ and the rational function field $H(T)$, respectively. See \S\ref{sssec:tiganoco_1.1} for more details.

In \cite[th\'eor\`eme A]{Beh21}, Behajaina shows that, for $\sigma$ of finite order, every finite group is a Galois group over $H(T, \sigma)$, if the fixed field of $\sigma$ in the center $h$ of $H$ contains an ample field, thus yielding a non-commutative analogue of Pop's result solving the regular inverse Galois problem over ample fields. Under weaker assumptions on $h$ (e.g., $h$ of characteristic zero), some ``usual" finite groups in inverse Galois theory, like abelian, symmetric, and alternating groups, are shown to occur as Galois groups over $H(T, \sigma)$ (see \cite[remarque du \S2.1]{Beh21}).

We combine Theorem \ref{thm:intro_1} and Behajaina's method to realize another family of usual finite groups as Galois groups over $H(T, \sigma)$, without making any assumption on the center (see \S\ref{ssec:tiganoco_2} for the proof of the next theorem). Recall that a finite group is {\it{semiabelian}} if it belongs to the smallest non-empty class $\mathcal{C}$ of finite groups closed under quotients and such that, if $G \in \mathcal{C}$ and $A$ is finite abelian, then every semidirect product $A \rtimes G$ is in $\mathcal{C}$.

\begin{theorem} \label{thm:intro_4}
Let $H$ be a division ring and $\sigma$ an automorphism of $H$ of finite order. Every finite semiabelian group is a Galois group over $H(T, \sigma)$.
\end{theorem}

We also contribute to solving finite embedding problems over division rings (a topic introduced in \cite{BDL20f}), by showing that every finite split embedding problem with abelian kernel over any given division ring $H$ of finite dimension over its center acquires a solution over the division ring $H(T)=H(T, {\rm{id}}_H)$. See \S\ref{sssec:tiganoco_1.2} for more on finite embedding problems over division rings and Theorem \ref{thm:main_2} for our precise result.

\vspace{2mm}

{\bf{Acknowledgements.}} Part of this work was done during a research visit at the Universit\'e Caen Normandie in March 2020. We would like to thank the Laboratoire de Math\'ematiques Nicolas Oresme for kind hospitality and financial support. Another part of this work fits into Project TIGANOCO, which is funded by the European Union within the framework of the Operational Programme ERDF/ESF 2014-2020. We also wish to thank Angelot Behajaina and Arno Fehm for helpful discussions about this work.

\section{Preliminaries} \label{sec:basics}

In this section, we collect the material on specializations of function field extensions, finite embedding problems over fields, and Krull valuations which will be used in the sequel.

\subsection{Specializations of function field extensions} \label{ssec:bas_1}

Let $k$ be a field, ${\bf{T}}=(T_1, \dots, T_n)$ a tuple of indeterminates, $E$ a Galois field extension of the rational function field $k({\bf{T}})$ of degree $d$, and ${B}$ the integral closure of $k[{\bf{T}}]$ in ${E}$. For ${\bf{t}}=(t_1, \dots, t_n) \in k^n$, the residue field of ${B}$ at a maximal ideal $\mathfrak{P}$ lying over $\langle T_1-t_1, \dots, T_n-t_n \rangle$ is denoted by ${E}_{\bf{t}}$, and the extension ${E}_{\bf{t}}/k$ is called the {\it{specialization}} of ${E}/k({\bf{T}})$ at ${\bf{t}}$. Since ${E}/k({\bf{T}})$ is Galois, the field ${E}_{\bf{t}}$ does not depend on the choice of the maximal ideal $\mathfrak{P}$ lying over $\langle T_1-t_1, \dots, T_n-t_n \rangle$. Moreover, the field extension ${E}_{\bf{t}}/k$ is finite and normal, and its automorphism group is isomorphic to $D_\mathfrak{P}/I_\mathfrak{P}$, where $D_\mathfrak{P}$ (resp., $I_\mathfrak{P}$) denotes the decomposition group (resp., the inertia group) of $E/k({\bf{T}})$ at $\mathfrak{P}$. More precisely,
\begin{equation} \label{eqvarphi}
\varphi_{{\bf{t}}}: \left \{ \begin{array} {ccc}
D_\mathfrak{P} & \rightarrow & {\rm{Aut}}({E}_{{\bf{t}}}/k) \\
\sigma & \mapsto & \overline{\sigma} \\
\end{array} \right. ,
\end{equation}
where $\overline{\sigma}(x \, {\rm{mod}} \, \mathfrak{P}) = \sigma(x) \, {\rm{mod}} \, \mathfrak{P}$ for every $x \in {B}$, is an epimorphism whose kernel is $I_\mathfrak{P}$. If ${E}_{\bf{t}}/k$ is Galois of degree $d$, then $\varphi_{\bf{t}}$ is an isomorphism ${\rm{Gal}}({E}/k({\bf{T}})) \rightarrow {\rm{Gal}}({E}_{{\bf{t}}}/k)$. For ${\bf{t}}$ outside a Zariski-closed proper subset (depending only on ${E}/k({\bf{T}})$), ${E}_{\bf{t}}/k$ is Galois.

\subsection{Finite embedding problems over fields} \label{ssec:bas_2}

Given finite Galois field extensions $L/k$ and $E/F$ with $k \subseteq F$ and $L \subseteq E$, we let ${\rm{res}}^{E/F}_{L/k}$ denote the natural restriction map ${\rm{Gal}}(E/F) \rightarrow {\rm{Gal}}(L/k)$ (i.e., ${\rm{res}}^{E/F}_{L/k}(\sigma)(x) = \sigma(x)$ for all $\sigma \in {\rm{Gal}}(E/F)$ and $x \in L$).

A {\it{finite embedding problem}} over a field $k$ is an epimor\-phism $\alpha : G \rightarrow {\rm{Gal}}(L/k)$, where $G$ is a finite group and $L/k$ a Galois field extension. Say that $\alpha$ {\it{splits}} if there is an embedding $\alpha' : {\rm{Gal}}(L/k) \rightarrow G$ with $\alpha \circ \alpha' = {\rm{id}}_{{\rm{Gal}}(L/k)}$. A {\it{solution}} to $\alpha$ is an isomorphism $\beta : {\rm{Gal}}(F/k) \rightarrow G$, where $F$ is a Galois field extension of $k$ containing $L$, such that $\alpha \circ \beta = {\rm{res}}^{F/k}_{L/k}$. Given a finite tuple ${\bf{T}}$ of indeterminates, the composed map
$$\alpha_{k({\bf{T}})}=({\rm{res}}^{L({\bf{T}})/k({\bf{T}})}_{L/k})^{-1} \circ \alpha  : G \rightarrow {\rm{Gal}}(L({\bf{T}})/k({\bf{T}}))$$
is a finite embedding problem over $k({\bf{T}})$. A {\it{geometric solution}} to $\alpha$ is a solution to $\alpha_{k(T)}$.

We now recall the notion of specializations at the level of finite embedding problems. To that end, we need the following lemma:

\begin{lemma} \label{lemma:spec_ffe}
Let ${\bf{T}}$ be an $n$-tuple of indeterminates, $\alpha : G \rightarrow {\rm{Gal}}(L/k)$ a finite embedding problem over a field $k$, and $\beta : {\rm{Gal}}(E/k({\bf{T}}))$ $\rightarrow G$ a solution to $\alpha_{k({\bf{T}})}$. For ${\bf{t}} \in k^n$ such that the specialization $E_{\bf{t}}/k$ of $E/k({\bf{T}})$ at ${\bf{t}}$ is Galois of degree $[E:k({\bf{T}})]$, the map $\beta \circ \varphi^{-1}_{{\bf{t}}} : {\rm{Gal}}(E_{{\bf{t}}}/k) \rightarrow G$ is a solution to $\alpha$, where $\varphi_{{\bf{t}}} : {\rm{Gal}}(E/k({\bf{T}})) \rightarrow {\rm{Gal}}(E_{{\bf{t}}}/k)$ is defined in \eqref{eqvarphi}.
\end{lemma}

\begin{proof}
As $\beta$ is a solution to $\alpha_{k({\bf{T}})}$, we have 
$$({\rm{res}}^{L({\bf{T}})/k({\bf{T}})}_{L/k})^{-1} \circ \alpha \circ \beta = {\rm{res}}^{E/k({\bf{T}})}_{L({\bf{T}})/k({\bf{T}})},$$ thus yielding 
$$\alpha \circ \beta \circ \varphi_{{\bf{t}}}^{-1} = {\rm{res}}_{L/k}^{L({\bf{T}})/k({\bf{T}})} \circ {\rm{res}}^{E/k({\bf{T}})}_{L({\bf{T}})/k({\bf{T}})} \circ \varphi_{{\bf{t}}}^{-1} = {\rm{res}}^{E/k({\bf{T}})}_{L/k} \circ \varphi_{{\bf{t}}}^{-1}.$$ 
It then suffices to show ${\rm{res}}^{E/k({\bf{T}})}_{L/k} = {\rm{res}}^{E_{{\bf{t}}}/k}_{L/k} \circ \varphi_{{\bf{t}}}$. But the latter equality holds as $\varphi_{{\bf{t}}}(\sigma)(x) = \sigma(x)$ for every $\sigma \in {\rm{Gal}}(E/k({\bf{T}}))$ and every $x \in L$, by the definition of $\varphi_{{\bf{t}}}$.
\end{proof}

The previous lemma motivates the following definition:

\begin{definition} \label{def:spec}
Let ${\bf{T}}$ be an $n$-tuple of indeterminates, $\alpha : G \rightarrow {\rm{Gal}}(L/k)$ a finite embedding problem over a field $k$, and $\beta : {\rm{Gal}}(E/k({\bf{T}}))$ $\rightarrow G$ a solution to $\alpha_{k({\bf{T}})}$. For ${\bf{t}} \in k^n$ such that the specialization $E_{\bf{t}}/k$ of $E/k({\bf{T}})$ at ${\bf{t}}$ is Galois of degree $[E:k({\bf{T}})]$, the solution $\beta \circ \varphi^{-1}_{{\bf{t}}} : {\rm{Gal}}(E_{{\bf{t}}}/k) \rightarrow G$ to $\alpha$ from Lemma \ref{lemma:spec_ffe} is the {\it{specialization}} of $\beta$ at ${\bf{t}}$, denoted $\beta_{\bf{t}}$.
\end{definition}

Finally, we observe that, to realize solutions as specialized solutions, it suffices to work at the level of fields:

\begin{lemma} \label{lemma:spec_ffe2}
Let $\alpha : G \rightarrow {\rm{Gal}}(L/k)$ be a finite embedding problem over a field $k$ and $\gamma : {\rm{Gal}}(F/k) \rightarrow G$ a solution to $\alpha$. Given an $n$-tuple ${\bf{T}}$ of indeterminates, let $E$ be a Galois field extension of $k({\bf{T}})$ containing $L$ with $[E:k({\bf{T}})] = |G|$, and let ${\bf{t}} \in k^n$ be such that the specialization $E_{\bf{t}}/k$ of $E/k({\bf{T}})$ at ${\bf{t}}$ is $F/k$. Then there is an isomorphism $\beta : {\rm{Gal}}(E/k({\bf{T}})) \rightarrow G$ which is a solution to $\alpha_{k({\bf{T}})}$ and such that the specialization $\beta_{\bf{t}}$ is $\gamma$.
\end{lemma}

\begin{proof}
As $E_{\bf{t}}= F$ and $[E:k({\bf{T}})] = |G|$, we have $[E_{\bf{t}}:k] = [E:k({\bf{T}})]$ and $E_{\bf{t}}/k$ is Galois. Hence, 
the map $\varphi_{\bf{t}}$ from \eqref{eqvarphi} is an isomorphism ${\rm{Gal}}({E}/k({\bf{T}})) \rightarrow {\rm{Gal}}({E}_{{\bf{t}}}/k)$. Using again $E_{\bf{t}}= F$, we get that $\gamma \circ \varphi_{{\bf{t}}} : {\rm{Gal}}(E/k({\bf{T}})) \rightarrow G$ is a well-defined isomorphism. Then $\beta = \gamma \circ \varphi_{{\bf{t}}}$ is a solution to $\alpha_{k({\bf{T}})}$ as $\alpha \circ \beta = {\rm{res}}^{F/k}_{L/k} \circ \varphi_{{\bf{t}}} = {\rm{res}}^{E_{\bf{t}}/k}_{L/k}  \circ \varphi_{{\bf{t}}} = {\rm{res}}^{E/k({\bf{T}})}_{L/k}$, the latter equality being observed in the proof of Lemma \ref{lemma:spec_ffe}. Finally, from our choice of $\beta$, the claimed equality $\beta_{\bf{t}} = \gamma$ is clear.
\end{proof}

\subsection{Valuations} \label{ssec:bas_3}

See, e.g., \cite{Jar91} for more on the following. An {\it{ordered group}} is an abelian (additive) group together with a total ordering $<$ such that $\alpha < \beta$ implies $\alpha + \gamma < \beta + \gamma$ for all $\alpha, \beta, \gamma \in \Gamma$. A (Krull) {\it{valuation}} of a field $k$ is a surjective map $v : k^* \rightarrow \Gamma$, with $\Gamma$ a non-trivial ordered group (the {\it{value group}} of $v$), such that $v(ab) = v(a) + v(b)$ for all $a, b \in k^*$ and $v(a + b) \geq {\rm{min}} \{v(a), v(b)\}$ for all $a, b \in k^*$ with $a+b \not=0$. Add the symbol $\infty$ to $\Gamma$ with the conventions $\infty + \infty = \alpha + \infty = \infty + \alpha = \infty$ and $\alpha < \infty$ for every $\alpha \in \Gamma$. Extend $v$ to $k$ by defining $v(0)=\infty$. Then $v$ fulfills $v(ab) = v(a) + v(b)$ and $v(a+b) \geq {\rm{min}} \{v(\alpha), v(\beta)\}$ for all $a, b \in k$. The pair $(k,v)$ is a {\it{valued field}}. Two valuations $v_1 : k^* \rightarrow \Gamma_1$ and $v_2 : k^* \rightarrow \Gamma_2$ of a field $k$ are {\it{equivalent}} if there is an order preserving isomorphism $f : \Gamma_1 \rightarrow \Gamma_2$ with $f \circ v_1 = v_2$.

Let $v : k^* \rightarrow \Gamma$ be a valuation of a field $k$. A valuation $w : L^* \rightarrow \Delta$ of a field extension $L$ of $k$ {\it{extends}} $v$ if there is an order preserving monomorphism $f : \Gamma \rightarrow \Delta$ such that $w(a) = f \circ v(a)$ for $a \in k^*$; say that $(L,w)$ is an {\it{extension}} of $(k,v)$. The valued field $(k,v)$ is {\it{Henselian}} if $v$ extends uniquely (up to equivalence) to each algebraic field extension of $k$. The valued field $(k,v)$ has a separable algebraic extension $(k_v^h, v^h)$ which is Henselian and such that, if $(L,w)$ is an extension of $(k,v)$ which is Henselian, then $(L,w)$ extends $(k_v^h, v^h)$. The valued field $(k_v^h, v^h)$ is unique up to $k$-isomorphism; it is the {\it{Henselization}} of $(k,v)$.

\section{Proof of Theorem \ref{thm:intro_1}} \label{sec:proof_1}

The aim of this section is to prove Theorem \ref{thm:intro_1} from the introduction (see \S\ref{ssec:proof_1.3}). As already alluded to in \S\ref{ssec:intro_1}, the proof requires several lemmas on ramification in higher dimension and specializations of function field extensions in several variables that we establish in \S\ref{ssec:proof_1.1} and \S\ref{ssec:proof_1.2}, respectively.

\subsection{Ramification in higher dimension} \label{ssec:proof_1.1}

Let $F/k$ be a finite Galois field extension. Suppose $k$ is the fraction field of an integrally closed Noetherian domain $A$. If there is no possible confusion, we will say that a $k$-basis of $F$ whose elements are integral over $A$ is an {\it{integral}} basis.  A maximal ideal $\mathfrak{p}$ of $A$ is {\it{unramified}} in $F/k$ if there exists an integral $k$-basis $y_1, \dots, y_d$ of $F$ whose discriminant $\Delta_{F/k}(y_1, \dots, y_d)$ fulfills $\Delta_{F/k}(y_1, \dots, y_d) \in A \setminus \mathfrak{p}$. 

\begin{lemma} \label{lem:ram_lin_disj}
Let $k$ be the fraction field of an integrally closed Noetherian domain $A$, let $\mathfrak{p}$ be a maximal ideal of $A$, and let $E/k$ be a finite Galois field extension. Let $F$ be a field extension of $k$ which is linearly disjoint from $E$ over $k$.

\vspace{0.5mm}

\noindent
{\rm{(1)}} Assume that $F/k$ is finite Galois, and that $\mathfrak{p}$ is unramified in both $E/k$ and $F/k$. Then $\mathfrak{p}$ is unramified in $EF/k$.

\vspace{0.5mm}

\noindent
{\rm{(2)}} Assume that $\mathfrak{p}$ is unramified in $E/k$, and let $\mathfrak{P}$ be a maximal ideal of the integral closure of $A$ in $F$ lying over $\mathfrak{p}$. Then $\mathfrak{P}$ is unramified in $EF/F$.
\end{lemma}

We need the following classical lemma about the discriminant of a composite basis.

\begin{lemma} \label{lemma:discriminant}
Let $F/k$ and $L/k$ be finite Galois field extensions with $L \subseteq F$. Let $y_1, \dots, y_d$ (resp., $z_1, \dots, z_e$) be a $k$-basis (resp., an $L$-basis) of $L$ (resp., of $F$). If $N_{L/k}$ denotes the norm in $L/k$, then the discriminant of the $k$-basis 
$y_1 z_1, \dots, y_1 z_e, \dots, y_d z_1, \dots, y_d z_e$ of $F$ equals $\Delta_{L/k}(y_1, \dots, y_d)^{e} \cdot N_{L/k}(\Delta_{F/L}(z_1, \dots, z_e)).$
\end{lemma}

\begin{proof}[Proof of Lemma \ref{lem:ram_lin_disj}]
We show (1) and (2) simultaneously. As $\mathfrak{p}$ is unramified in $E/k$ (in both statements), there is an inte\-gral $k$-basis $y_1, \dots, y_e$ of $E$ such that $\Delta_{E/k}(y_1, \dots, y_e) \in A \setminus \mathfrak{p}$. Moreover, $E$ and $F$ are linearly disjoint over $k$. Then $y_1, \dots, y_e$ yield an integral $F$-basis of $EF$, and the discriminant of the latter basis cannot be in $\mathfrak{P}$. Hence, $\mathfrak{P}$ is unramified in $EF/F$, thus establishing (2). Now, in the situation of (1), there is an integral $k$-basis $z_1, \dots, z_f$ of $F$ such that $\Delta_{F/k}(z_1, \dots, z_f) \in A \setminus \mathfrak{p}$. From our assumption that $E$ and $F$ are linearly disjoint over $k$, the elements $y_1z_1, \dots, y_1z_f, \dots, y_e z_1, \dots, y_ez_f$ provide an integral $k$-basis of $EF$. Lemma \ref{lemma:discriminant} then yields that the discriminant $\Delta$ of the latter basis equals
$$\begin{array} {ccl}
 \Delta_{F/k}(z_1, \dots, z_f)^e \cdot N_{F/k}(\Delta_{EF/F}(y_1, \dots, y_e)) & = & \Delta_{F/k}(z_1, \dots, z_f)^e \cdot  N_{F/k}(\Delta_{E/k}(y_1, \dots, y_e)) \\
 & = & \Delta_{F/k}(z_1, \dots, z_f)^e \cdot \Delta_{E/k}(y_1, \dots, y_e)^f.
\end{array}$$
Since neither $\Delta_{E/k}(y_1, \dots, y_e)$ nor $\Delta_{F/k}(z_1, \dots, z_f)$ is in $\mathfrak{p}$, we get $\Delta \not \in \mathfrak{p}$, as needed.
\end{proof}

\subsection{Specialization lemmas} \label{ssec:proof_1.2}

We start with specializations of composite field extensions:

\begin{lemma} \label{compositum}
Let $E/k({\bf{T}}) = E/k(T_1, \dots, T_n)$ be a finite Galois field extension and let ${\bf{t}}=(t_1, \dots, t_n) \in k^n$ be such that $\langle T_1-t_1, \dots, T_n-t_n \rangle$ is unramified in $E/k({\bf{T}})$.

\vspace{0.5mm}

\noindent
{\rm{(1)}} Let $F/k({\bf{T}})$ be a finite Galois field extension such that $E$ and $F$ are linearly disjoint over $k({\bf{T}})$. Then the specialization $(EF)_{{\bf{t}}}/k$ of $EF/k({\bf{T}})$ at ${\bf{t}}$ equals $E_{\bf{t}}F_{\bf{t}}/k$, where $E_{\bf{t}}/k$ (resp., $F_{\bf{t}}/k$) is the specialization of $E/k({\bf{T}})$ (resp., of $F/k({\bf{T}})$) at ${\bf{t}}$.

\vspace{0.5mm}

\noindent
{\rm{(2)}} Let $U$ be a new indeterminate. Then the specialization $(E(U))_{(u, {\bf{t}})}/k$ of $E(U)/k(U,{\bf{T}})$ at $(u,{\bf{t}})$ equals the specialization $E_{\bf{t}}/k$ of $E/k({\bf{T}})$ at ${\bf{t}}$ for every $u \in k$.
\end{lemma}

\begin{proof}
(1) Let $\mathfrak{P}$ be a maximal ideal of the integral closure of $k[{\bf{T}}]$ in $E$ containing $T_1 - t_1, \dots, T_n-t_n$. As $\langle T_1-t_1, \dots, T_n-t_n \rangle$ is unramified in $E/k({\bf{T}})$, there is an integral $k({\bf{T}})$-basis $y_1, \dots, y_e$ of $E$ such that $\Delta({\bf{t}}) \not=0$, where $\Delta({\bf{T}}) \in k[{\bf{T}}]$ is the discriminant of $y_1, \dots, y_e$. Let $B$ denote the integral closure of $k[{\bf{T}}]$ in $EF$. As $E$ and $F$ are linearly disjoint over $k({\bf{T}})$, the elements $y_1, \dots, y_e$ yield an integral $F$-basis of $EF$. Now, given $x \in B$, there are $\lambda_1, \dots, \lambda_e \in F$ with $x = \lambda_1 y_1 + \cdots + \lambda_e y_e.$ For $i \in \{1, \dots, e\}$, we then have 
$${\rm{tr}}_{EF/F}(xy_i) = \lambda_1 {\rm{tr}}_{EF/F}(y_1y_i) + \cdots + \lambda_e {\rm{tr}}_{EF/F}(y_ey_i),$$ 
with ${\rm{tr}}_{EF/F}$ the trace in $EF/F$. As $x, y_1, \dots, y_e \in B$, we get that $\Delta({\bf{T}}) \lambda_i$ is in the integral closure of $k[{\bf{T}}]$ in $F$ for every $i \in \{1, \dots, e\}$, by applying Kramer's law. Hence, if $\mathfrak{Q}$ denotes a maximal ideal of $B$ lying over $\mathfrak{P}$, then the reduction modulo $\mathfrak{Q}$ of 
$$x = \lambda_1 y_1 + \cdots + \lambda_e y_e = \frac{\Delta({\bf{T}}) \lambda_1}{\Delta({\bf{T}})} y_1 + \cdots + \frac{\Delta({\bf{T}}) \lambda_e}{\Delta({\bf{T}})} y_e$$
lies in the compositum of $E_{\bf{t}}$ and $F_{\bf{t}}$, thus showing $(EF)_{\bf{t}} \subseteq E_{\bf{t}} F_{\bf{t}}$. As to the converse, it suffices to notice that $E \subseteq EF$ and $F \subseteq EF$ yield $E_{\bf{t}} \subseteq (EF)_{\bf{t}}$ and $F_{\bf{t}} \subseteq (EF)_{\bf{t}}$.

\vspace{1mm}

\noindent
(2) As above, there is an integral $k({\bf{T}})$-basis $y_1, \dots, y_e$ of $E$ with $\Delta({\bf{t}}) \not=0$, where $\Delta({\bf{T}}) \in k[{\bf{T}}]$ is the discriminant of $y_1, \dots, y_e$. As $E$ and $k({\bf{T}},U)$ are linearly disjoint over $k({\bf{T}})$, the elements $y_1, \dots, y_e$ provide a $k({\bf{T}},U)$-basis of $E(U)$, which is contained in the integral closure $B$ of $k[U, {\bf{T}}]$ in $E(U)$. Moreover, we have $\Delta({\bf{T}}) B \subseteq k[U, {\bf{T}}] {y_1} \oplus \cdots \oplus k[U, {\bf{T}}] {y_e}.$ Let $u \in k$. As $\Delta({\bf{T}})$ does not depend on $U$ and $\Delta({\bf{t}}) \not=0$, the above inclusion yields that $(E(U))_{(u, {\bf{t}})}$ is the field generated over $k$ by the reductions of $y_1, \dots, y_e$ modulo a maximal ideal of $B$ lying over $\langle U-u, T_1-t_1, \dots, T_n-t_n \rangle$, i.e., equals $E_{\bf{t}}$.
\end{proof}

The next lemma deals with successive specializations:

\begin{lemma} \label{je_sais_plus}
Let $E/k({\bf{T}}) = E/k(T_1, \dots, T_n)$ be a finite Galois field extension, $t_1 \in k$, and ${\bf{t'}}=(t_2(T_1), \dots, t_n(T_1)) \in k[T_1]^{n-1}$ such that the specialization $E_{\bf{t'}}/k(T_1)$ of $E/k(T_1)(T_2, \dots, T_n)$ at ${\bf{t'}}$ is Galois. Then the specialization $E_{\bf{t}}/k$ of $E/k({\bf{T}})$ at ${\bf{t}} = (t_1, t_2(t_1), \dots, t_n(t_1))$ is a subextension of the specialization $(E_{\bf{t'}})_{t_1}/k$ of $E_{\bf{t'}}/k(T_1)$ at $t_1$.
\end{lemma}

\begin{proof}
Denote the integral closure of $k(T_1)[T_2, \dots, T_n]$ (resp., of $k[{\bf{T}}]$) in $E$ by $B_2$ (resp., by $B$). Clearly, $B \subseteq B_2$. For a maximal ideal $\mathfrak{P}_2$ of $B_2$ lying over $\langle T_2-t_2(T_1) , \dots, T_n-t_n(T_1) \rangle$ and $x \in B_2$, we denote the reduction of $x$ modulo $\mathfrak{P}_2$ by $\overline{x}$. Then the reduction $\overline{x}$ of any element $x$ of $B$ is integral over $k[T_1]$, i.e., is an element of the integral closure $B_1$ of $k[T_1]$ in $E_{\bf{t'}}= B_2/\mathfrak{P}_2$. Let $\mathfrak{P}_1$ be a maximal ideal of $B_1$ containing $T_1-t_1$. Composing the inclusion $B/(B \cap \mathfrak{P}_2) \subseteq B_1$ with the canonical projection $B_1 \rightarrow B_1/\mathfrak{P}_1$ yields a well-defined homomorphism $\psi : B \rightarrow B_1/\mathfrak{P}_1$ satisfying $\langle T_1 - t_1, T_2 - t_2(t_1), \dots, T_n-t_n(t_1) \rangle \subseteq {\rm{ker}}(\psi) \cap k[{\bf{T}}].$ Hence, $\langle T_1 - t_1, T_2 - t_2(t_1), \dots, T_n-t_n(t_1)  \rangle = {\rm{ker}}(\psi) \cap k[{\bf{T}}],$ as the ideal in the left-hand side is maximal and $k[{\bf{T}}] \not \subseteq {\rm{ker}}(\psi)$. Consequently, the ideal ${\rm{ker}}(\psi)$ of $B$ lies over $\langle T_1 - t_1, T_2 - t_2(t_1), \dots, T_n-t_n(t_1) \rangle$ and it is maximal. Hence, $E_{\bf{t}} = B/{\rm{ker}}(\psi) \subseteq B_1/\mathfrak{P}_1 = (E_{\bf{t'}})_{t_1}$.
\end{proof}

Our last lemma gives a situation where the inclusion of Lemma \ref{je_sais_plus} is actually an equality:

\begin{lemma} \label{lemma2}
Let $E/k({\bf{T}}) = E/k(T_1, \dots, T_n)$ be a finite Galois field extension, $t_1 \in k$, and ${\bf{t'}}=(t_2(T_1), \dots, t_n(T_1)) \in k[T_1]^{n-1}$. Suppose the following three conditions hold:

\vspace{0.5mm}

\noindent
$\bullet$ $\langle T_2-t_2(T_1) , \dots, T_n-t_n(T_1) \rangle$ is unramified in $E/k(T_1)(T_2, \dots, T_n)$, 

\vspace{0.5mm}

\noindent
$\bullet$ the specialization $E_{\bf{t'}}/k(T_1)$ of $E/k(T_1)(T_2, \dots, T_n)$ at ${\bf{t'}}$ is Galois, 

\vspace{0.5mm}

\noindent
$\bullet$ $t_1$ is outside some finite set $\mathcal{S}$ depending on $E/k({\bf{T}})$ and ${\bf{t'}}$. 

\vspace{0.5mm}

\noindent
Then the specialization $E_{\bf{t}}/k$ of $E/k({\bf{T}})$ at ${\bf{t}} = (t_1, t_2(t_1), \dots, t_n(t_1))$ equals the specialization $(E_{\bf{t'}})_{t_1}/k$ of $E_{\bf{t'}}/k(T_1)$ at $t_1$.
\end{lemma}

\begin{proof}
Denote the integral closure of $k(T_1)[T_2, \dots, T_n]$ (resp., of $k[{\bf{T}}]$) in $E$ by $B_2$ (resp., by $B$). For a maximal ideal $\mathfrak{P}_2$ of $B_2$ lying over $\langle T_2-t_2(T_1) , \dots, T_n-t_n(T_1) \rangle$ and $x \in B_2$, we denote the reduction of $x$ modulo $\mathfrak{P}_2$ by $\overline{x}$. As $\langle T_2-t_2(T_1) , \dots, T_n-t_n(T_1) \rangle$ is unramified in $E/k(T_1)(T_2, \dots, T_n)$, there exist $y_1, \dots, y_e \in B_2$ such that $E=k({\bf{T}}) y_1 \oplus \cdots \oplus k({\bf{T}}) y_e,$
$$\Delta(y_1, \dots, y_e) B_2 \subseteq k(T_1)[T_2, \dots, T_n] {y_1} \oplus \cdots \oplus k(T_1)[T_2, \dots, T_n] {y_e},$$
and such that the discriminant $\Delta(y_1, \dots, y_e) \in k(T_1)[T_2, \dots, T_n]$ of $y_1, \dots, y_e$ is not in the ideal $\langle T_2-t_2(T_1), \dots, T_n-t_n(T_1) \rangle$. In particular, we have
\begin{equation} \label{eqpoly2}
E_{\bf{t'}} = B_2/\mathfrak{P}_2 = k(T_1) \frac{\overline{y_1}}{\overline{\Delta(y_1, \dots, y_e)}} + \cdots + k(T_1) \frac{\overline{y_e}}{\overline{\Delta(y_1, \dots, y_e)}}.
\end{equation}
Pick $b(T_1) \in k[T_1] \setminus \{0\}$ such that $z_i=b(T_1) y_i$ is in $B$ for every $i \in \{1, \dots, e\}$. Then, from \eqref{eqpoly2}, we have $E_{\bf{t'}} = k(T_1) \overline{z_1} + \cdots +k(T_1) \overline{z_e}.$ Up to reordering, there exists $c \leq e$ such that 
\begin{equation} \label{eqpoly3}
E_{\bf{t'}} = k(T_1) \overline{z_1} \oplus \cdots \oplus k(T_1) \overline{z_c}.
\end{equation}
Let $\Delta(T_1) \in k[T_1] \setminus \{0\}$ be the discriminant of $\overline{z_1}, \dots, \overline{z_c}$ and let $\mathcal{S}$ be the set of all roots of $\Delta(T_1)$. For $t_1 \in k \setminus \mathcal{S}$ and a maximal ideal $\mathfrak{P}_1$ of the integral closure $B_1$ of $k[T_1]$ in $E_{\bf{t'}}$ containing $T_1-t_1$, \eqref{eqpoly3} and $t_1 \not \in \mathcal{S}$ yield that $(E_{\bf{t'}})_{t_1} = B_1/\mathfrak{P}_1$ is the field generated over $k$ by the reductions of $\overline{z_1}, \dots, \overline{z_c}$ modulo $\mathfrak{P}_1$. Since these reductions are necessarily in $E_{\bf{t}}$, we get $(E_{\bf{t'}})_{t_1} \subseteq E_{\bf{t}}$. Lemma \ref{je_sais_plus} then yields
$(E_{\bf{t'}})_{t_1} = E_{\bf{t}}$, thus concluding the proof.
\end{proof}

\begin{remark} \label{rem:finite}
Under the extra assumptions that $[E_{{\bf{t'}}}:k(T_1)] = [E:k({\bf{T}})]$ and $y_1, \dots, y_e$ are integral over $k[{\bf{T}}]$, the set $\mathcal{S}$ is reduced to the set of all roots of $\overline{\Delta(y_1, \dots, y_e)} \in k[T_1]$.
\end{remark}

\subsection{Proof of Theorem \ref{thm:intro_1}} \label{ssec:proof_1.3}

Since $\alpha$ splits, there exists an embedding $\alpha' : {\rm{Gal}}(L/k) \rightarrow G$ satisfying 
\begin{equation} \label{eqid}
\alpha \circ \alpha' = {\rm{id}}_{{\rm{Gal}}(L/k)}.
\end{equation} Set $G' = \alpha'({\rm{Gal}}(L/k)) \subseteq G$, and $A={\rm{ker}}(\alpha) \subseteq G$. Then
\begin{equation} \label{eqdefpsi}
\psi: \left \{ \begin{array} {ccc}
A \rtimes G' & \longrightarrow & G \\
(a, g') & \longmapsto & a \cdot g'
\end{array} \right. ,
\end{equation}
where $G'$ acts by conjugation on $A$, is an isomorphism. Let $c_1, \dots, c_n$ be a $k$-basis of $L$ and set ${\rm{Gal}}(L/k) = \{\sigma_1, \dots, \sigma_n\}$.

Let $T$ be an extra indeterminate, and let $F/k(T)$ be a Galois extension of group $A$ in which $\langle T \rangle$ is unramified and whose specialization $F_0/k$ at 0 equals $k/k$; such an extension exists by the twisting lemma (see \cite{Deb99a}). From the unramifiedness condition, there is a $k(T)$-basis $z_1, \dots, z_{|A|}$ of $F$, which is integral over $k[T]$ and such that $\Delta(0) \in k^*$, where $\Delta(T) \in k[T]$ denotes the discriminant of $z_1, \dots, z_{|A|}$. Then consider the Galois field extension $E/L(T) = FL/L(T)$. As $F \cap \overline{k}=k$, we have ${\rm{Gal}}(E/L(T)) = A$. Let $P(T,Y) \in L[T][Y]$ be the minimal polynomial of a primitive element of $E/L(T)$, assumed to be integral over $L[T]$. Given $i \in \{1, \dots, n\}$, let $U_i$ be a new indeterminate and $E^{(i)}$ the field generated over $L(U_i)$ by one root, say $y_i$, of $\sigma_i(P)(U_i,Y) \in L[U_i][Y]$, where $\sigma_i(P)(U_i,Y)$ is obtained from $P(U_i,Y)$ by applying $\sigma_i$ to the coefficients. We then let $M^{(i)}$ denote the field $E^{(i)}(U_1, \dots, U_{i-1}, U_{i+1}, \dots  U_n)$ and set $M = M^{(1)} \cdots M^{(n)}$. The extension $M/L({\bf{U}})= M/L(U_1, \dots, U_n)$ is Galois of group $A^n$. 

\begin{lemma} \label{lemma:fj1}
{\rm{(1)}} There is an $L({\bf{U}})$-basis of $M$ which is integral over $L[{\bf{U}}]$ and whose discri\-mi\-nant equals $(\Delta(U_1)$ $\cdots \Delta(U_n))^{|A|^{n-1}}$ for some $\Delta(T) \in k[T]$ with $\Delta(0) \in k^*$. In particular, $\langle U_1, \dots, U_n \rangle$ is unramified in $M/L({\bf{U}})$.

\vspace{1mm}

\noindent
{\rm{(2)}} The specialization $M_{(0, \dots, 0)}/L$ of $M/L({\bf{U}})$ at $(0, \dots, 0)$ equals $L/L$.
\end{lemma}

\begin{proof}
We work with the $k(T)$-basis $z_1, \dots, z_{|A|}$ of $F$ from the above. Since $F \cap \overline{k}=k$, the elements $z_1, \dots, z_{|A|}$ yield an $L(T)$-basis of $E$ which is integral over $L[T]$ and whose discriminant equals $\Delta(T)$. In particular, $\langle T \rangle$ is unramified in $E/L(T)$. By Lemma \ref{compositum}(1), the specialization $E_0/L$ of $E/L(T)$ at 0 is $L/L$. Next, let $i=1, \dots, n$. The elements $\sigma_i(z_1), \dots, \sigma_i(z_{|A|})$ of $E^{(i)}$ are an $L(U_i)$-basis of $E^{(i)}$ which is integral over $L[U_i]$. Its discriminant $\Delta_i(U_i)$ is obtained from $\Delta(T)$ by replacing $T$ by $U_i$ and by applying $\sigma_i$ to the coefficients. But, since $\Delta(T) \in k[T]$, we have $\Delta_i(U_i) = \Delta(U_i)$. In particular, $\Delta_i(0) \not= 0$ and $\langle U_i \rangle$ is unramified in $E^{(i)}/L(U_i)$. Moreover, the specialization $(E^{(i)})_0/L$ of $E^{(i)}/L(U_i)$ at 0 equals $L/L$. Lemma \ref{lem:ram_lin_disj}(2) then shows that $\langle U_1, \dots, U_n \rangle$ is unramified in $M^{(i)}/L({\bf{U}})$. Moreover, by Lemma \ref{compositum}(2), the specialization $(M^{(i)})_{(0, \dots, 0)}/L$ of $M^{(i)}/L({\bf{U}})$ at $(0, \dots, 0)$ equals $L/L$. Finally, by Lemma \ref{lem:ram_lin_disj}(1), $\langle U_1, \dots, U_n \rangle$ is unramified in $M/L({\bf{U}})$. Then the $L({\bf{U}})$-basis $\{\sigma_1(z_{i_1}) \cdots \sigma_n(z_{i_n})\}_{1 \leq i_1, \dots, i_n \leq |A|}$ of $M$ is integral over $L[{\bf{U}}]$. By the cons\-truction and Lemma \ref{lemma:discriminant}, its discriminant equals $(\Delta(U_1) \cdots \Delta(U_n))^{|A|^{n-1}}$, thus proving (1). As to (2), as $\langle U_1, \dots, U_n \rangle$ is unramified in $M^{(i)}/L({\bf{U}})$ for each $i \in \{1, \dots, n\}$, Lemma \ref{compositum}(1) yields $M_{(0,\dots,0)}=(M^{(1)})_{(0, \dots, 0)} \cdots (M^{(n)})_{(0, \dots, 0)}$, i.e.,  $M_{(0,\dots,0)}=L$. This completes the proof.
\end{proof}

Now, there is a unique tuple ${\bf{T}}=(T_1, \dots, T_n)$ of indeterminates fulfilling:
\begin{equation} \label{eqsystem}
\left \{\begin{array} {ccc}
\sigma_1(c_1) T_1 + \sigma_1(c_2) T_2 + \cdots + \sigma_1(c_n)T_n & = & U_1 \\
\dots & =& \dots \\
\sigma_n(c_1) T_1 + \sigma_n(c_2) T_2 + \cdots + \sigma_n(c_n)T_n & = & U_n 
\end{array} \right..
\end{equation}

\begin{lemma} \label{lemma:fj2}
{\rm{(1)}} The extension $M/k({\bf{T}})$ is Galois.

\vspace{0.5mm}

\noindent
{\rm{(2)}} The specialization of $M/k({\bf{T}})$ at $(0, \dots, 0)$ equals $L/k$.

\vspace{0.5mm}

\noindent
{\rm{(3)}} The maximal ideal $\langle T_1, \dots, T_n \rangle$ of $k[{\bf{T}}]$ is unramified in $M/k({\bf{T}})$.
\end{lemma}

\begin{proof}
(1) The extensions $M/L({\bf{T}})$ and $L({\bf{T}})/k({\bf{T}})$ are Galois. Given $i \in \{1, \dots, n\}$, it then suffices to find $\tau_i \in {\rm{Aut}}(M/k({\bf{T}}))$ such that the restriction of $\tau_i$ to $L$ is $\sigma_i$. To that end, note that the multiplication on the left by $\sigma_i$ yields a pemutation $\epsilon_{i}$ of $\{1, \dots, n\}$. Then consider the automorphism $\tau_i : M \rightarrow M$ defined by $\tau_i(U_j)=U_{\epsilon_{i}(j)}$ and $\tau_i(y_j)= y_{\epsilon_{i}(j)}$ for $j \in \{1, \dots, n\}$, and $\tau_i|{_{ L}} = \sigma_i$. The automorphism $\tau_i$ permutes the equations of \eqref{eqsystem}, i.e., $(\tau_i(T_1), \dots, \tau_i(T_n))$ is a solution to \eqref{eqsystem}. As the solution to \eqref{eqsystem} is unique, we have $\tau_i(T_i) = T_i$ for every $i \in \{1, \dots, n\}$. Hence, $\tau_i$ is in ${\rm{Aut}}(M/k({\bf{T}}))$ and extends $\sigma_i$.

\vspace{1.5mm}

\noindent
(2) By the definition of ${\bf{T}}$, the integral closure of $L[{\bf{U}}]$ in $M$ equals the integral closure of $L[{\bf{T}}]$ in $M$, and a maximal ideal of the latter integral closure contains $T_1, \dots, T_n$ if and only if it contains $U_1, \dots, U_n$. The claim then follows from Lemma \ref{lemma:fj1}(2).

\vspace{1.5mm}

\noindent
(3) By Lemma \ref{lemma:fj1}(1), there is an integral $L({\bf{U}})$-basis $b_1, \dots, b_{|A|^n}$ of $M$ whose discrimi\-nant is $(\Delta(U_1) \cdots \Delta(U_n))^{|A|^{n-1}}$ for some $\Delta(T) \in k[T]$ with $\Delta(0) \in k^*$. Clearly, $b_1, \dots, b_{|A|^n}$ yield an integral $L({\bf{T}})$-basis of $M$ whose discriminant $\widetilde{\Delta}({\bf{T}})$, as an element of $L[{\bf{T}}]$, equals 
$$\widetilde{\Delta}({\bf{T}})=\prod_{i=1}^n (\Delta(\sigma_i(c_1) T_1  + \cdots  + \sigma_i(c_n) T_n))^{|A|^{n-1}}.$$
As applying any element of ${\rm{Gal}}(L({\bf{T}})/k({\bf{T}}))$ to $\widetilde{\Delta}({\bf{T}})$ permutes the factors, we have $\widetilde{\Delta}({\bf{T}}) \in k[{\bf{T}}]$, and $\widetilde{\Delta}({\bf{T}})$ does not vanish at $(0, \dots, 0)$. Moreover, $c_1, \dots, c_n$ yield a $k({\bf{T}})$-basis of $L({\bf{T}})$, which is integral over $k[{\bf{T}}]$ and whose discriminant, which is in $k^*$, cannot be in $\langle T_1, \dots, T_n \rangle$. Lemma \ref{lemma:discriminant} then shows that the discriminant of the integral $k({\bf{T}})$-basis $\{b_ic_j\}_{1 \leq i \leq {|A|^n}, 1 \leq j \leq n}$ of $M$ is not in $\langle T_1, \dots, T_n \rangle$. Hence, $\langle T_1, \dots, T_n \rangle$ is unramified in $M/k({\bf{T}})$.
\end{proof}

Given groups $G_1$ and $G_2$, recall that ${G_1}^{G_2}$ denotes the group of all functions $f : G_2 \rightarrow G_1$, and that $G_2$ acts on ${G_1}^{G_2}$ by $f^{\tau}(\sigma) = f(\tau \sigma)$ ($\tau \in G_2, \sigma \in G_2$). The corresponding semidirect product will be denoted by ${G_1}^{G_2} \rtimes G_2$, and the projection map ${G_1}^{G_2} \rtimes G_2 \rightarrow G_2$ by ${\rm{pr}}$.

The first part of the following lemma is a special case of \cite[Remarks 13.7.6 and 13.7.7]{FJ08}. As to the second part, it is \cite[Lemma 16.4.3]{FJ08}.

\begin{lemma} \label{lemma:fj3}
{\rm{(1)}} There exists an isomorphism $\phi : {\rm{Gal}}(M/k({\bf{T}})) \rightarrow A^{G'} \rtimes G'$ satisfying 
\begin{equation} \label{eqlemma}
{\rm{pr}} \circ \phi = \alpha' \circ {\rm{res}}^{M/k({\bf{T}})}_{L/k}.
\end{equation}

\noindent
{\rm{(2)}} The following map is an epimorphism:
\begin{equation} \label{eqlemma2}
\varphi : \left \{ \begin{array} {ccc}
A^{G'} \rtimes G' & \longrightarrow & A \rtimes G'\\
(f, \alpha'(\sigma)) & \longmapsto &  (f(\alpha' (\sigma_1))^{\alpha'(\sigma_1)^{-1}} \cdots f(\alpha'(\sigma_n))^{\alpha'(\sigma_n)^{-1}} , \alpha'(\sigma))
\end{array}. \right. 
\end{equation}
\end{lemma}

Let us next consider $N=M^{{\rm{ker}}(\psi \circ \varphi \circ \phi)},$ where $\psi$, $\phi$, and $\varphi$ are defined in \eqref{eqdefpsi}, \eqref{eqlemma}, and \eqref{eqlemma2}, respectively. By Lemma \ref{lemma:fj3}(1), ${\rm{Gal}}(M/L({\bf{T}})) = \phi^{-1}(A^{G'} \rtimes \{1\})$. Hence, $L \subseteq N$ by the definition of $\varphi$. Moreover, by Lemma \ref{lemma:fj2}(2), the specialization $N_{(0, \dots, 0)}/k$ of $N/k({\bf{T}})$ at $(0, \dots, 0)$ equals $L/k$. Finally, there is an isomorphism $\beta : {\rm{Gal}}(N/k({\bf{T}})) \rightarrow G$ satisfying
\begin{equation} \label{eqbetafin}
\beta \circ {\rm{res}}^{M/k({\bf{T}})}_{N/k({\bf{T}})} = \psi \circ \varphi \circ \phi.
\end{equation}
We then have $\alpha \circ \beta={\rm{res}}^{N/k({\bf{T}})}_{L/k}$. Indeed, given $\sigma \in {\rm{Gal}}(N/k({\bf{T}}))$, let $\tau \in {\rm{Gal}}(M/k({\bf{T}}))$ be such that ${\rm{res}}^{M/k({\bf{T}})}_{N/k({\bf{T}})}(\tau) = \sigma$. Since \eqref{eqlemma2} gives $\psi \circ \varphi (A^{G'} \times \{1\}) \subseteq A$ ($= {\rm{ker}}(\alpha)$), we have:
$$\alpha \circ \beta(\sigma) \stackrel{(\ref{eqbetafin})}{=} \alpha \circ \psi \circ \varphi \circ \phi(\tau) \stackrel{(\ref{eqlemma2})}{=}  \alpha \circ \psi (1, {\rm{pr}} \circ \phi(\tau)) \stackrel{(\ref{eqlemma})}{=} \alpha \circ \psi (1, \alpha' \circ {\rm{res}}^{M/k({\bf{T}})}_{L/k}(\tau))$$
$$\hspace{8.5cm} \stackrel{(\ref{eqdefpsi})}{=} \alpha \circ \alpha' \circ {\rm{res}}^{M/k( {\bf{T}})}_{L/k}(\tau)$$
$$\hspace{7.1cm} \stackrel{(\ref{eqid})}{=} {\rm{res}}^{M/k({\bf{T}})}_{L/k}(\tau)$$
$$ \hspace{7.1cm}  ={\rm{res}}^{N/k({\bf{T}})}_{L/k}(\sigma).$$

\begin{lemma} \label{prop:2}
Let $k$ be a Hilbertian field, $\mathcal{S}$ a finite set of valuations of $k$, and $\alpha : G \rightarrow {\rm{Gal}}(L/k)$ a finite embedding problem over $k$. Assume there are a tuple ${\bf{T}} = (T_1, \dots, T_n)$ of indeterminates and a solution $\beta: {\rm{Gal}}(N/k({\bf{T}})) \rightarrow G$ to $\alpha_{k({\bf{T}})}$ such that the ideal $\langle T_1, \dots, T_n \rangle$ is unra\-mi\-fied in $N/k({\bf{T}})$. Then the specialization $N_{(0, \dots, 0)}/k$ of $N/k({\bf{T}})$ at $(0, \dots, 0)$ is Galois and there is a solution ${\rm{Gal}}(F/k) \rightarrow G$ to $\alpha$ such that $F \cdot k_v^h = N_{(0, \dots,0)} \cdot k_v^h$ for every $v \in \mathcal{S}$.
\end{lemma}

\noindent
{\it{Addendum}} \ref{prop:2}. The field $F$ can be chosen as follows. Given a new in\-de\-ter\-mi\-nate $T$, consider the Galois field extension $N(T)$ of $k(T,{\bf{T}})$. Then there exist ${\bf{a}} \in k[T]^n$ and $t \in k$ such that

\vspace{0.5mm}

\noindent
$\bullet$ the specialization $(N(T))_{\bf{a}}/k(T)$ of $N(T)/k(T)({\bf{T}})$ at ${\bf{a}}$ is Galois of group $G$,

\vspace{0.5mm}

\noindent
$\bullet$ the specialization $((N(T))_{\bf{a}})_{t}/k$ of $(N(T))_{\bf{a}}/k(T)$ at $t$ is Galois of group $G$,

\vspace{0.5mm}

\noindent
$\bullet$ $F=((N(T))_{\bf{a}})_{t}$.

\begin{proof}
We break the proof into five parts.

\vspace{2mm}

\noindent
(a) Let $P({\bf{T}},Y) \in k[{\bf{T}}][Y]$ be the minimal polynomial of a primitive element of $N/k({\bf{T}})$, assumed to be integral over $k[{\bf{T}}]$. As $k$ is assumed to be Hilbertian, there is ${\bf{t}} = (t_1, \dots, t_n) \in k^n$ such that $P({\bf{t}},Y) \in k[Y]$ is irreducible over $k$ and separable. Hence, the specialization $N_{\bf{t}}/k$ of $N/k({\bf{T}})$ at ${\bf{t}}$ is Galois of degree $[N:k({\bf{T}})]$ and $N_{\bf{t}}$ is the splitting field over $k$ of $P({\bf{t}},Y)$. Given a new indeterminate $T$, for $i \in \{1, \dots, n\}$, fix $a_i(T) \in k[T]$ with $a_i(0)=0$ and $a_i(1)=t_i$. Consider the Galois extension $N(T)/k(T)({\bf{T}})$ of group $G$ and its specialization $(N(T))_{\bf{a}} / k(T)$ at ${\bf{a}}= (a_1(T), \dots, a_n(T))$. As $P({\bf{t}},Y)$ is separable and $a_i(1)=t_i$ for every $i \in \{1, \dots, n\}$, the polynomial $P({\bf{a}},Y)$ is separable. Hence, $(N(T))_{\bf{a}}$ is the splitting field over $k(T)$ of $P({\bf{a}},Y)$ and $(N(T))_{\bf{a}}/k(T)$ is separable. Using again that $P({\bf{t}},Y)$ is separable and $a_i(1)=t_i$ for every $i \in \{1, \dots, n\}$, we get that $((N(T))_{\bf{a}})_1$ is the splitting field over $k$ of $P({\bf{t}},Y)$, i.e., equals $N_{\bf{t}}$. As $[N_{\bf{t}}:k]=[N(T):k(T, {\bf{T}})]$, we get the equality $[(N(T))_{\bf{a}}:k(T)] = [N(T):k(T)({\bf{T}})]$.

\vspace{2mm}

\noindent
(b) The maximal ideal $\langle T \rangle$ is unramified in $(N(T))_{\bf{a}}/k(T)$. Indeed, as $\langle T_1, \dots, T_n \rangle$ is unra\-mi\-fied in $N/k({\bf{T}})$, there is a $k({\bf{T}})$-basis $y_1, \dots, y_d$ of $N$ which is integral over $k[{\bf{T}}]$ and whose discriminant $\Delta({\bf{T}}) \in k[{\bf{T}}]$ fulfills $\Delta(0, \dots, 0) \in k^*$. As $k({\bf{T}}, T)$ and $N$ are linearly disjoint over $k({\bf{T}})$, the elements $y_1, \dots, y_d$ yield a $k({\bf{T}}, T)$-basis of $N(T)$ which is contained in the integral closure $B$ of $k(T)[{\bf{T}}]$ in $N(T)$ and with $\Delta({\bf{T}}) B \subseteq k(T)[{\bf{T}}] y_1 \oplus \cdots \oplus k(T)[{\bf{T}}] y_d.$ As $\Delta(a_1(0), \dots, a_n(0)) = \Delta(0, \dots, 0) \not=0$, we have $\Delta({\bf{a}}) \not=0$ (in particular, $\langle T_1 - a_1(T), \dots, T_n - a_n(T) \rangle$ is unramified in $N(T)/k(T)({\bf{T}})$). Hence, if $\mathfrak{P}$ is a maximal ideal of $B$ lying over $\langle T_1 - a_1(T), \dots, T_n - a_n(T) \rangle$, we have $(N(T))_{\bf{a}} = B/\mathfrak{P} = k(T) \overline{y_1} \oplus \cdots \oplus k(T) \overline{y_d},$ where $\overline{y_i}$ is the reduction modulo $\mathfrak{P}$ of $y_i$ ($1 \leq i \leq d$). As the discriminant $\Delta(a_1(T), \dots, a_n(T))$ of $\overline{y_1}, \dots, \overline{y_d}$ does not vanish at 0, the claim holds.

\vspace{2mm}

\noindent
(c) The specialization $((N(T))_{\bf{a}})_{0}/k$ of $(N(T))_{\bf{a}}/k(T)$ at 0 equals $N_{(0, \dots, 0)}/k$ (in particular, $N_{(0, \dots, 0)}/k$ is Galois by (b)). Indeed, as $\langle T_1, \dots, T_n \rangle$ is unramified in $N/k({\bf{T}})$, Lemma \ref{compositum}(2) yield that the specialization $N_{(0,0, \dots, 0)}/k$ of $N(T)/k(T, {\bf{T}})$ at $(0,0, \dots, 0)$ equals $N_{(0, \dots, 0)}/k$. Moreover, $\langle T_1 - a_1(T), \dots, T_n - a_n(T) \rangle$ is unramified in $N(T)/k(T)({\bf{T}})$ (see (b)). The claim then follows from Lemma \ref{lemma2}, if 0 is not in the exceptional set from that lemma, when applied to $N(T)/k(T, T_1, \dots, T_n)$. But, in the present situation, the latter set is reduced to the set of roots of $\Delta(a_1(T), \dots, a_n(T))$ (see Remark \ref{rem:finite}), which does not contain 0 (see (b)).

\vspace{2mm}

\noindent
(d) We now specialize $T$ suitably. As $\langle T \rangle$ is unramified in $(N(T))_{\bf{a}}/k(T)$ (see (b)) and $k$ is infinite, there is a primitive ele\-ment of $(N(T))_{\bf{a}}$ over $k(T)$, which is integral over $k[T]$ and whose minimal polynomial $Q(T,Y) \in k[T][Y]$ is such that $Q(0,Y) \in k[Y]$ is separable (see, e.g., \cite[corollaire 1.5.16]{Deb09}). Krasner's lemma and the ``continuity of roots" (see, e.g., \cite[Proposition 12.3]{Jar91}), applied to the separable polynomial $Q(0,Y)$ over the Henselian valued field $(k_v^h, v^h)$, provide an element $\alpha_v$ of the value group of $v$ such that, if $t \in k$ fulfills $v(t) \geq \alpha_v$, then the splitting fields of $Q(0,Y)$ and $Q(t,Y)$ over $k_v^h$ coincide.

Next, as $k$ is Hilbertian, \cite[Proposition 19.6]{Jar91} yields $t \in k$ such that $v(t) \geq \alpha_v$ for every $v \in \mathcal{S}$, and such that $Q(t, Y) \in k[Y]$ is irreducible over $k$ and separable. Let $F$ be the field generated over $k$ by one root of $Q(t, Y)$. As the latter is irreducible over $k$, we have $[F:k] = [(N(T))_{\bf{a}}:k(T)]$ and $F/k$ is the specialization $((N(T))_{\bf{a}})_{t}/k$ of $(N(T))_{\bf{a}}/k(T)$ at $t$. In particular, $((N(T))_{\bf{a}})_{t}/k$ is Galois of degree $[((N(T))_{\bf{a}}:k(T)]$.

Finally, as $Q(0, Y)$ is separable, the splitting field of $Q(0, Y)$ over $k$ equals $((N(T))_{\bf{a}})_0$, i.e., equals $N_{(0, \dots, 0)}$ (see (c)). Hence, for $v \in \mathcal{S}$, we have $N_{(0, \dots, 0)} \cdot k_v^h = F \cdot k_v^h$.

\vspace{2mm}

\noindent
(e) Finally, as $\beta : {\rm{Gal}}(N / k({\bf{T}})) \rightarrow G$ is a solution to $\alpha_{k({\bf{T}})}$, the map $\beta \circ {\rm{res}}^{N(T)/k(T, {\bf{T}})}_{N/k({\bf{T}})} : {\rm{Gal}}(N(T) / k(T, {\bf{T}})) \rightarrow G$ is a solution to $\alpha_{k(T, {\bf{T}})} = (\alpha_{k(T)})_{k(T)({\bf{T}})}$. Since $(N(T))_{\bf{a}} /k(T)$ is a Galois field extension of degree $[N(T) :k(T)({\bf{T}})]$, Lemma \ref{lemma:spec_ffe} shows that $(\beta \circ {\rm{res}}^{N(T)/k(T, {\bf{T}})}_{N/k({\bf{T}})})_{\bf{a}} : {\rm{Gal}}((N(T))_{\bf{a}}/k(T)) \rightarrow G$ is a solution to $\alpha_{k(T)}$. Moreover, $F = ((N(T))_{\bf{a}})_{t}$ is a Galois field  extension of $k$ of degree $[(N(T))_{\bf{a}} : k(T)]$. Applying Lemma \ref{lemma:spec_ffe} once more then shows that $((\beta \circ {\rm{res}}^{N(T)/k(T, {\bf{T}})}_{N/k({\bf{T}})})_{\bf{a}})_t : {\rm{Gal}}(F/k) \rightarrow G$ is a solution to $\alpha$. This concludes the proof.
\end{proof}

By Lemma \ref{prop:2} and the above, to conclude the proof of Theorem \ref{thm:intro_1}, it would suffice to show that the ideal $\langle T_1, \dots, T_n \rangle$ is unramified in $N/k({\bf{T}})$. As proved above, $\langle T_1, \dots, T_n \rangle$ is unramified in $M/k({\bf{T}})$ and $N \subseteq M$. However, with our definition of ``unramified", which is designed for computing residue fields, it is not clear that unramifiedness is closed under subextensions. To conclude the proof of Theorem \ref{thm:intro_1}, we proceed as follows.
 
With $\Gamma = {\rm{Gal}}(M/k({\bf{T}}))$, consider the finite embedding problem ${\rm{res}}^{M/k({\bf{T}})}_{L/k} : \Gamma \rightarrow {\rm{Gal}}(L/k)$ over $k$. Clearly, ${\rm{id}} : {\rm{Gal}}(M/k({\bf{T}})) \rightarrow \Gamma$ is a solution to $({\rm{res}}^{M/k({\bf{T}})}_{L/k})_{k({\bf{T}})}$. By Lemma \ref{lemma:fj2}, we may apply Lemma \ref{prop:2} and its addendum at the level of $M$. Namely, given a new indeterminate $T$, there exist ${\bf{a}} \in k[T]^n$, $t \in k$, and a solution ${\rm{Gal}}(F/k) \rightarrow \Gamma$ to ${\rm{res}}^{M/k({\bf{T}})}_{L/k}$ such that the following conditions hold:

\noindent
$\bullet$ the specialization $(M(T))_{\bf{a}}/k(T)$ of $M(T)/k(T)({\bf{T}})$ at ${\bf{a}}$ is Galois of group $\Gamma$,

\noindent
$\bullet$ the specialization $((M(T))_{\bf{a}})_{t}/k$ of $(M(T))_{\bf{a}}/k(T)$ at $t$ is Galois of group $\Gamma$,

\noindent
$\bullet$ $F=((M(T))_{\bf{a}})_{t}$ and $F \subseteq L \cdot k_v^h$ for every $v \in \mathcal{S}$.

Finally, we proceed as in the last part of the proof of Lemma \ref{prop:2}. Consider the solution $\beta: {\rm{Gal}}(N/k({\bf{T}})) \rightarrow G$ to $\alpha_{k({\bf{T}})}$. Then 
$$\beta \circ {\rm{res}}^{N(T)/k(T, {\bf{T}})}_{N/k({\bf{T}})} : {\rm{Gal}}(N(T)/k(T,{\bf{T}})) \rightarrow G$$ 
is a solution to $(\alpha_{k(T)})_{k(T)({\bf{T}})}$. As $(M(T))_{\bf{a}}/k(T)$ is Galois of degree $[M(T) : k(T)({\bf{T}})]$, the specialization $(N(T))_{\bf{a}}/k(T)$ of $N(T)/k(T)({\bf{T}})$ at ${\bf{a}}$, which is a subextension of $(M(T))_{\bf{a}}/k(T)$, is Galois of degree $[N(T) : k(T)({\bf{T}})]$. Lemma \ref{lemma:spec_ffe} then yields a solution ${\rm{Gal}}((N(T))_{\bf{a}}/k(T)) \rightarrow G$ to $\alpha_{k(T)}$. Again, $((M(T))_{\bf{a}})_{t}/k$ is Galois of degree $[(M(T))_{\bf{a}}:k(T)]$. Hence, $((N(T))_{\bf{a}})_{t}/k$ is Galois of degree $[(N(T))_{\bf{a}}:k(T)]$. Lemma \ref{lemma:spec_ffe} then yields a solution ${\rm{Gal}}(((N(T))_{\bf{a}})_{t}/k) \rightarrow G$ to $\alpha$ such that $((N(T))_{\bf{a}})_{t} \subseteq ((M(T))_{\bf{a}})_{t} \subseteq L \cdot k_v^h$ for every $v \in \mathcal{S}$, thus concluding the proof.

\section{On the Beckmann--Black problem} \label{sec:proof_2}

The following theorem, which solves the Beckmann--Black problem for finite embedding problems with abelian kernels over arbitrary fields, is the aim of the present section:

\begin{theorem} \label{thm:intro_2}
Let $\alpha: G \rightarrow {\rm{Gal}}(L/k)$ be a finite embedding problem with abelian kernel over any field $k$ and $\gamma : {\rm{Gal}}(F/k) \rightarrow G$ a solution to $\alpha$. There are $t_0 \in k$ and a geometric solution $\beta : {\rm{Gal}}(E/k(T)) \rightarrow G$ to $\alpha$ with $E \cap \overline{k}=L$ such that $E_{t_0}/k = F/k$ and $\beta_{t_0} = \gamma$.
\end{theorem}

The present section is organized as follows. In \S\ref{ssec:proof_2.1}, we prove Theorem \ref{thm:intro_2} under the extra assumption that $\alpha$ splits, and then show in \S\ref{ssec:proof_2.2} that the latter assumption is redundant. Finally, in \S\ref{ssec:proof_2.3}, we derive a more precise version of Theorem \ref{thm:intro_3}.

\subsection{Proof of Theorem \ref{thm:intro_2} under an extra assumption} \label{ssec:proof_2.1}

\begin{proposition} \label{prop:split}
Let $\alpha :G \rightarrow {\rm{Gal}}(L/k)$ be a finite split embedding problem with abelian kernel over an arbitrary field $k$ and $\gamma : {\rm{Gal}}(F/k) \rightarrow G$ a solution to $\alpha$. There are a geometric solution $\beta :{\rm{Gal}}(E/k(T)) \rightarrow G$ to $\alpha$ with $E \cap \overline{k} = L$, such that the specialization $E_0/k$ of $E/k(T)$ at $0$ equals $F/k$, and such that the specialization $\beta_{0}$ of $\beta$ at $0$ equals $\gamma$.
\end{proposition}

\begin{proof}
First, as $\alpha$ splits, there is an embedding $\alpha' : {\rm{Gal}}(L/k) \rightarrow G$ with $\alpha \circ \alpha' = {\rm{id}}_{{\rm{Gal}}(L/k)}$. Set $G' = \alpha({\rm{Gal}}(L/k))$. Then $G'$ acts on ${\rm{ker}}(\alpha)$ by conjugation and 
$$\phi: \left \{ \begin{array} {ccc}
{\rm{ker}}(\alpha) \rtimes G' & \longrightarrow & G \\
(a, g') & \longmapsto & a \cdot g'
\end{array} \right. $$
is an isomorphism. Moreover, since $\alpha$ splits and has abelian kernel, we may apply Theorem \ref{thm:intro_1} to get the existence of a geometric solution $\beta : {\rm{Gal}}(E/k(T)) \rightarrow G$ to $\alpha$ such that the inclusion $E \subseteq L((T))$ holds.

Now, we identify ${\rm{Gal}}(EF/k(T))$ and a subgroup of $({\rm{ker}}(\alpha) \rtimes G') \times ({\rm{ker}}(\alpha) \rtimes G')$. Since $E \cap \overline{k}=L$, the fields $E$ and $F(T)$ are linearly disjoint over $L(T)$. Hence, with
$$\widetilde{G} = \{ (\sigma, \tau) \in {\rm{Gal}}(E/k(T)) \times {\rm{Gal}}(F(T)/k(T)) \, : \, {\rm{res}}^{E/k(T)}_{L(T)/k(T)} (\sigma) = {\rm{res}}^{F(T)/k(T)}_{L(T)/k(T)}(\tau)\},$$
$$ \psi : \left \{ \begin{array} {ccc}
{\rm{Gal}}(EF/k(T)) & \longrightarrow & \widetilde{G} \\
\sigma & \longmapsto & ({\rm{res}}^{EF/k(T)}_{E/k(T)}(\sigma), {\rm{res}}^{EF/k(T)}_{F(T)/k(T)}(\sigma))
\end{array} \right. $$
is an isomorphism. Next, $\phi^{-1} \circ \beta : {\rm{Gal}}(E/k(T)) \rightarrow {\rm{ker}}(\alpha) \rtimes G'$ and $\phi^{-1} \circ \gamma \circ {\rm{res}}^{F(T)/k(T)}_{F/k} : {\rm{Gal}}(F(T)/k(T)) \rightarrow {\rm{ker}}(\alpha) \rtimes G'$ are isomorphisms. Then consider the monomorphism
$$ \theta : \left \{ \begin{array} {ccc}
\widetilde{G} & \longrightarrow & ({\rm{ker}}(\alpha) \rtimes G') \times ({\rm{ker}}(\alpha) \rtimes G')  \\
(\sigma, \tau) & \longmapsto & (\phi^{-1} \circ \beta(\sigma), \phi^{-1} \circ \gamma \circ {\rm{res}}^{F(T)/k(T)}_{F/k}(\tau))
\end{array} \right. .$$
We claim that the image of $\theta$ equals $\mathcal{G}= \{((a,g'),(b,g')) \, : \, (a,b) \in {\rm{ker}}(\alpha) \times {\rm{ker}}(\alpha), \, g' \in G'\}$. Hence, from the above, the composed map $\theta \circ \psi : {\rm{Gal}}(EF/k(T)) \rightarrow \mathcal{G}$ is an isomorphism.

To show the claim, for $(\sigma, \tau) \in {\rm{Gal}}(E/k(T)) \times {\rm{Gal}}(F(T)/k(T))$, set $\phi^{-1} \circ \beta (\sigma) = (a, g'_1) \in {\rm{ker}}(\alpha) \rtimes G'$ and $\phi^{-1} \circ \gamma \circ {\rm{res}}^{F(T)/k(T)}_{F/k}(\tau) = (b, g'_2) \in {\rm{ker}}(\alpha) \rtimes G'$. Set $g'_i = \alpha'(g_i)$ with $g_i \in G$ for $i=1, 2$. We then have
$$g_1 = \alpha \circ \alpha'(g_1) = \alpha(g'_1) = \alpha(ag'_1) = \alpha \circ \beta(\sigma) = {\rm{res}}^{E/k(T)}_{L/k}(\sigma) = {\rm{res}}^{L(T)/k(T)}_{L/k} ({\rm{res}}^{E/k(T)}_{L(T)/k(T)}(\sigma))$$
and, similarly, $g_2= {\rm{res}}^{L(T)/k(T)}_{L/k} ({\rm{res}}^{F(T)/k(T)}_{L(T)/k(T)}(\tau)).$ Hence, the image of $\theta$ equals $\mathcal{G}$, as claimed.

Finally, consider the subgroup $\mathcal{H} = \{((a,1),(a,1)) \, : \,  a \in {\rm{ker}}(\alpha)\}$ of $\mathcal{G}$. As ${\rm{ker}}(\alpha)$ is a\-be\-lian, $\mathcal{H}$ is normal in $\mathcal{G}$. Hence, with $M = (EF)^{(\theta \circ \psi)^{-1}(\mathcal{H})}$, the extension $M/k(T)$ is Galois of degree $[EF:k(T)] / |{\rm{ker}}(\alpha)| = |G|$. Moreover, as $\theta \circ \psi ({\rm{Gal}}(EF/L(T))) = \{((a,1),(b,1)) \, : \,  (a,b) \in {\rm{ker}}(\alpha) \times {\rm{ker}}(\alpha)\}$, we have $L(T) \subseteq M$. Furthermore, as $M$ and $E$ are linearly disjoint over $L(T)$, and as $[M:L(T)] = |{\rm{ker}}(\alpha)|$, we get $EF=EM$. Hence, the specialization $(EF)_0/k$ of $EF/k(T)$ at 0 is the specialization $(EM)_0/k$ of $EM/k(T)$ at 0. But, as $E \subseteq L((T))$, the ideal $\langle T \rangle$ is unramified in $E/k(T)$. By, e.g., \cite[Lemma 2.4.8]{FJ08} (or Lemma \ref{compositum}(1)), we then have $(EF)_0 = E_0 (F(T))_0 = E_0 F$ and $(EM)_0 = E_0 M_0$. Using again $E \subseteq L((T))$, we get $E_0 = L$ and so $F=M_0$. It then remains to apply Lemma \ref{lemma:spec_ffe2} to conclude the proof of the proposition, up to showing $M \cap \overline{k} = L$. But $M \overline{k} \supseteq MF = EF$, i.e., $E \overline{k} \subseteq M \overline{k}$. As $[E\overline{k} : \overline{k}(T)] = [E : L(T)] = [M:L(T)]$, we get $[M \overline{k} : \overline{k}(T)] = [M:L(T)]$, thus showing the desired equality $M \cap \overline{k}=L$.
\end{proof}

\subsection{Proof of Theorem \ref{thm:intro_2}} \label{ssec:proof_2.2}

Denote the canonical projection $G \rightarrow G/{\rm{ker}}(\alpha)$ by $\varphi$. Set $G_\varphi^2 = \{(g_1, g_2)) \in G^2 \, : \, \varphi(g_1)= \varphi(g_2)\}$ and, for $i \in \{1,2\}$, consider
$$\pi_i :\left  \{ \begin{array} {ccc}
G_\varphi^2 & \longrightarrow & G\\
(g_1, g_2) & \longmapsto & g_i
\end{array} \right. .$$
We also denote the kernel of $\pi_i$ by $N_i$ for $i \in \{1,2\}$. 

Since we are given the solution $\gamma$ to $\alpha$, whose kernel is abelian, we may combine the {\it{weak$\rightarrow$split reduction}} (see \cite[\S1 B) 2)]{Pop96} and \cite[\S2.1.2]{DD97b}) and, e.g., \cite[\S16.4]{FJ08} to get that $\alpha$ has a geometric solution $\delta : {\rm{Gal}}(F_2/k(T_1)) \rightarrow G$ fulfilling $F_2 \cap \overline{k} = L$.

\begin{lemma} \label{lemmafin1}
{\rm{(1)}} The following map is a well-defined isomorphism:
$$\epsilon: \left \{ \begin{array} {ccc}
{\rm{Gal}}(FF_2/k(T_1)) & \longrightarrow & G_\varphi^2\\
\sigma & \longmapsto & (\gamma \circ {\rm{res}}^{FF_2/k(T_1)}_{F/k}(\sigma), \delta \circ {\rm{res}}^{FF_2/k(T_1)}_{F_2/k(T_1)}(\sigma))
\end{array} \right. .$$

\noindent
{\rm{(2)}} We have $\epsilon ({\rm{Gal}}(FF_2/L(T_1))) = {\rm{ker}}(\alpha) \times {\rm{ker}}(\alpha)$.
\end{lemma}

\begin{proof}
Let ${\theta} : {\rm{Gal}}(L/k) \rightarrow G/{\rm{ker}}(\alpha)$ be the unique isomorphism satisfying ${\theta}  \circ \alpha = \varphi$. As
$$ \varphi \circ \gamma \circ {\rm{res}}^{FF_2/k(T_1)}_{F/k}= \theta \circ \alpha \circ \gamma \circ {\rm{res}}^{FF_2/k(T_1)}_{F/k} = \theta \circ {\rm{res}}^{F/k}_{L/k} \circ {\rm{res}}^{FF_2/k(T_1)}_{F/k} = \theta \circ {\rm{res}}^{FF_2/k(T_1)}_{L/k},$$
$$\varphi \circ \delta \circ {\rm{res}}^{FF_2/k(T_1)}_{F_2/k(T_1)}= \theta \circ \alpha \circ \delta \circ {\rm{res}}^{FF_2/k(T_1)}_{F_2/k(T_1)} = \theta \circ {\rm{res}}_{L/k}^{F_2/k(T_1)} \circ {\rm{res}}^{FF_2/k(T_1)}_{F_2/k(T_1)} = \theta \circ {\rm{res}}^{FF_2/k(T_1)}_{L/k},$$
$\epsilon$ is well-defined. Moreover, $\epsilon$ is clearly injective and, since $F_2 \cap \overline{k} = L$, the fields $F(T_1)$ and $F_2$ are linearly disjoint over $L(T_1)$, thus showing that $\epsilon$ is also surjective. Hence, (1) holds. As for (2), for $\sigma \in {\rm{Gal}}(FF_2/L(T_1))$, we have $\alpha \circ \gamma \circ {\rm{res}}^{FF_2/k(T_1)}_{F/k}(\sigma) = {\rm{res}}^{F/k}_{L/k} \circ {\rm{res}}^{FF_2/k(T_1)}_{F/k}(\sigma) = {\rm{id}}_{{\rm{Gal}}(L/k)}$ and, similarly, $\alpha \circ  \delta \circ {\rm{res}}^{FF_2/k(T_1)}_{F_2/k(T_1)}(\sigma)={\rm{id}}_{{\rm{Gal}}(L/k)}$, thus showing $\epsilon({\rm{Gal}}(FF_2/L(T_1))) \subseteq {\rm{ker}}(\alpha) \times {\rm{ker}}(\alpha)$. Since both groups have the same order, we are done.
\end{proof}

Now, consider the finite embedding problem $\delta^{-1} \circ \pi_2 : G_\varphi^2 \rightarrow {\rm{Gal}}(F_2/k(T_1))$ over $k(T_1)$. Then $\delta^{-1} \circ \pi_2$ splits and its kernel $N_2$ is abelian. Moreover, the isomorphism $\epsilon$ from Lemma \ref{lemmafin1}(1) is a solution to $\delta^{-1} \circ \pi_2$. Consequently, we may use Proposition \ref{prop:split} to lift the solution $\epsilon$ to a geometric solution. Namely, there exists a geometric solution $\beta : {\rm{Gal}}(E/k(T_1)(T_2)) \rightarrow G_\varphi^2$ to $\delta^{-1} \circ \pi_2$ with $E \cap \overline{k(T_1)} = F_2$, such that the specialization $E_0/k(T_1)$ of $E/k(T_1)(T_2)$ at 0 equals $FF_2/k(T_1)$, and such that the specialization of $\beta$ at $0$ is $\epsilon$.

\begin{lemma} \label{lemmafin4}
{\rm{(1)}} We have $E^{\beta^{-1}(N_1)} \cap \overline{k(T_1)} = L(T_1)$.

\vspace{0.5mm}

\noindent
{\rm{(2)}} The specialization $(E^{\beta^{-1}(N_1)})_0/k(T_1)$ of $E^{\beta^{-1}(N_1)}/k(T_1)(T_2)$ at $0$ equals $F(T_1)/k(T_1)$.
\end{lemma}

\begin{proof}
Let $\varphi_{0} : {\rm{Gal}}(E/k(T_1)(T_2)) \rightarrow {\rm{Gal}}(E_{0}/k(T_1))$ be the isomorphism defined in \eqref{eqvarphi}. 

\vspace{1mm}

\noindent
(1) By the definition of $\varphi_0$, we have
$$\varphi_{0}({\rm{Gal}}(E/L(T_1)(T_2))) = {\rm{Gal}}(E_{0}/L(T_1)) = {\rm{Gal}}(FF_2/L(T_1))  = \epsilon^{-1}({\rm{ker}}(\alpha) \times {\rm{ker}}(\alpha)),$$
the last equality being Lemma \ref{lemmafin1}(2). As $\beta \circ \varphi_0^{-1} = \epsilon$, we get ${\rm{Gal}}(E/L(T_1,T_2)) = \beta^{-1}({\rm{ker}}(\alpha) \times {\rm{ker}}(\alpha))$ and, as $N_1 \subseteq {\rm{ker}}(\alpha) \times {\rm{ker}}(\alpha)$, we get $L(T_1, T_2) \subseteq E^{\beta^{-1}(N_1)}$. In particular, $L(T_1) \subseteq E^{\beta^{-1}(N_1)} \cap \overline{k(T_1)}$. 

For the converse, the equality $E \cap \overline{k(T_1)} = F_2$ yields $E^{\beta^{-1}(N_1)} \cap \overline{k(T_1)} = E^{\beta^{-1}(N_1)} \cap F_2.$ Hence, $(E^{\beta^{-1}(N_1)} \cap \overline{k(T_1)})(T_2)$ is a subfield of
$$E^{\beta^{-1}(N_1)} \cap F_2(T_2)= E^{\beta^{-1}(N_1)} \cap E^{\beta^{-1}(N_2)} {=}  E^{\langle \beta^{-1}(N_1), \beta^{-1}(N_2)\rangle} = E^{\beta^{-1}({\rm{ker}}(\alpha) \times {\rm{ker}}(\alpha))},$$
i.e., a subfield of $L(T_1, T_2)$. We then have $E^{\beta^{-1}(N_1)} \cap \overline{k(T_1)} \subseteq L(T_1)$.

\vspace{1.5mm}

\noindent
(2) By the definition of $\epsilon$, we have $F(T_1)= (FF_2)^{\epsilon^{-1}(N_1)} = (E_{0})^{\epsilon^{-1}(N_1)} = (E_0)^{\varphi_0 \circ \beta^{-1}(N_1)}.$ Hence, $(E^{\beta^{-1}(N_1)})_0 = F(T_1)$ if and only if $(E^{\beta^{-1}(N_1)})_{0} = (E_0)^{\varphi_{0} \circ \beta^{-1}(N_1)},$ i.e., if and only if ${\rm{Gal}}(E_{0}/(E^{\beta^{-1}(N_1)})_{0}) = {\varphi_{0} \circ \beta^{-1}(N_1)}$. As the latter equality holds by the definition of $\varphi_{0}$, we have $(E^{\beta^{-1}(N_1)})_0 = F(T_1)$, as needed.
\end{proof}

Next, as in the fourth part of the proof of Lemma \ref{prop:2}, we specialize $T_2$ suitably. By Lemma \ref{lemmafin4}(2), the ideal $\langle T_2 \rangle$ is unramified in $E^{\beta^{-1}(N_1)}/L(T_1)(T_2)$. Then, as $L(T_1)$ is infinite, there exists a primitive element of $E^{\beta^{-1}(N_1)}$ over $L(T_1)(T_2)$ which is integral over $L(T_1)[T_2]$ and whose minimal polynomial $P(T_2,Y) \in L(T_1)[T_2][Y]$ is such that $P(0,Y) \in L(T_1)[Y]$ is separable. Krasner's lemma and the ``continuity of roots" (see, e.g., \cite[Proposition 12.3]{Jar91}), applied to the separable polynomial $P(0,Y)$ over the complete valued field $L((T_1))$, yield a positive integer $N$ such that, for every $t_2(T_1) \in L(T_1)$ of valuation at least $N$ with respect to $\langle T_1 \rangle$, the splitting fields of $P(0,Y)$ and $P(t_2(T_1), Y)$ over $L((T_1))$ coincide.

Finally, by Lemma \ref{lemmafin4}(1), we have $E^{\beta^{-1}(N_1)} \cap \overline{k} = L$. Hence, $P(T_2,Y)$ is irreducible over $\overline{k}(T_1)(T_2)$. In particular, $P(T_1^N T_2, Y)$ and $P(T_1^N T_2^{-1}, Y)$ are also irreducible over $\overline{k}(T_1)(T_2)$. Then apply either \cite[Proposition 13.2.1]{FJ08} if $k$ is infinite or \cite[Theorem 13.4.2 and Proposition 16.11.1]{FJ08} if $k$ is finite to get the existence of $t_2(T_1) \in k(T_1)$ such that $P(T_1^N t_2(T_1), Y)$ $\in L(T_1)[Y]$ and $P(T_1^N t_2(T_1)^{-1}, Y) \in L(T_1)[Y]$ are irreducible over $\overline{k}(T_1)$ and separable. Without loss, we may assume that $t_2(T_1)$ is of non-negative valuation with respect to $\langle T_1 \rangle$. Let $M$ be the field generated over $L(T_1)$ by one root of $P(T_1^N t_2(T_1), Y)$. As the latter is irreducible over $L(T_1)$, one has $[M:k(T_1)] = [E^{\beta^{-1}(N_1)}: k(T_1)(T_2)]$ and, consequently, the fields $M$ and $(E^{\beta^{-1}(N_1)})_{t_2(T_1)}$ coincide. It then remains to use the irreducibility of $P(T_1^N t_2(T_1), Y)$ over $\overline{k}(T_1)$ to get $(E^{\beta^{-1}(N_1)})_{t_2(T_1)} \cap \overline{k} = L$. Moreover, since $P(T_1^N t_2(T_1), Y)$ is separable, its splitting field over $L(T_1)$ equals $(E^{\beta^{-1}(N_1)})_{t_2(T_1)}$. Similarly, as $P(0,Y)$ is se\-pa\-rable, its splitting field over $L(T_1)$ equals $(E^{\beta^{-1}(N_1)})_{0}$, i.e., equals $F(T_1)$ (see Lemma \ref{lemmafin4}(2)). Hence, from the previous paragraph, the completion of $(E^{\beta^{-1}(N_1)})_{t_2(T_1)}$ with respect to any prime extending $\langle T_1 \rangle$ equals $F((T_1))$. In particular, the specialization $((E^{\beta^{-1}(N_1)})_{t_2(T_1)})_0/k$ of $(E^{\beta^{-1}(N_1)})_{t_2(T_1)}/k(T_1)$ at $0$ equals $F/k$. It then remains to apply Lemma \ref{lemma:spec_ffe2} to conclude the proof of Theorem \ref{thm:intro_2}.

\subsection{Proof of Theorem \ref{thm:intro_3}} \label{ssec:proof_2.3}

We actually have the following more precise consequence:

\begin{corollary} \label{thm0}
Let $k$ be an arbitrary field, $G$ a finite group, and $F/k$ a Galois extension of group $G$. Assume $G$ has a non-trivial solvable normal subgroup. Then there exist $t_0 \in k$ and a Galois field extension $E/k(T)$ of group $G$, with $E \not \subset \overline{k}(T)$, and such that $F/k = E_{t_0}/k$.
\end{corollary}

\begin{proof}
By the assumption, $G$ has a non-trivial solvable normal subgroup $H$. The smallest non-trivial term $H'$ of the derived series of $H$ then is an abelian characteristic subgroup of $H$, and so is a non-trivial abelian normal subgroup of $G$. Let $\theta : G \rightarrow {\rm{Gal}}(F/k)$ be an isomorphism. Consider the finite embedding problem 
$$\alpha={\rm{res}}^{F/k}_{F^{\theta(H')}/k} \circ \theta : G \rightarrow {\rm{Gal}}(F^{\theta(H')}/k)$$
over $k$; it has kernel $H'$, which is abelian, and $\theta^{-1} : {\rm{Gal}}(F/k) \rightarrow G$ is a solution. Theorem \ref{thm:intro_2} then yields $t_0 \in k$ and a geometric solution ${\rm{Gal}}(E/k(T)) \rightarrow G$ to $\alpha$ with $E_{t_0}/k = F/k$ and $E \cap \overline{k}= F^{\theta(H')}$. From the latter, we have $[E \cap \overline{k}:k]=|G/H'|$ and, as $H'$ is not trivial, this implies $E \not \subset \overline{k}(T)$, thus concluding the proof.
\end{proof}

\begin{remark}
(1) For $n \geq 1$, let $D_{2^n}$ be the dihedral group of cardinality $2^n$. The Beckmann--Black problem for $D_{2^n}$ over $\Qq$ is known to have a positive answer only for $n \leq 5$ (see \cite[th\'eor\`eme 2.2]{Deb01b} for references), and no counter-example is known. Our method gives that, for every field $k$, every $n \geq 1$, and every Galois extension $F/k$ of group $D_{2^n}$, there are $t_0 \in k$ and a Galois extension $E/k(T)$ of group $D_{2^n}$, with $[E \cap \overline{k}:k]=2$, and such that $F/k = E_{t_0}/k$.

\vspace{1mm}

\noindent
(2) Say that a finite group $G$ is a {\it{non-constant Galois group}} over a field $k$ if there exists a Galois field extension $E/k(T)$ of group $G$ with $E \not \subset \overline{k}(T)$. Corollary \ref{thm0} allows to reobtain that, if a finite group $G$ is the Galois group of a Galois field extension of $k$ and if $G$ has a non-trivial solvable normal subgroup, then $G$ is a non-constant Galois group over $k$.
\end{remark}

\section{Inverse Galois theory over division rings} \label{sec:tiganoco}

The present section is organized as follows. In \S\ref{ssec:tiganoco_1}, we collect some material from previous papers about division rings and finite embedding problems in this context. We then prove Theorem \ref{thm:intro_4} in \S\ref{ssec:tiganoco_2}. Finally, in \S\ref{ssec:tiganoco_3}, we present our result about finite embedding problems with abelian kernels over division rings alluded to at the end of \S\ref{sssec:intro_2.2}.

\subsection{Preliminaries} \label{ssec:tiganoco_1}

We refer to \cite{DL20, Beh21, BDL20f} for more on the following.

\subsubsection{Division rings} \label{sssec:tiganoco_1.1}

Let $H \subseteq L$ be division rings. The group of all automorphisms of $L$ fixing $H$ pointwise is the {\it{automorphism group}} ${\rm{Aut}}(L/H)$ of $L/H$. Following Artin, we say that $L/H$ is {\it{Galois}} if each element of $L$ which is fixed under all elements of ${\rm{Aut}}(L/H)$ is in $H$. If $L/H$ is Galois, ${\rm{Aut}}(L/H)$ is the {\it{Galois group}} ${\rm{Gal}}(L/H)$ of $L/H$.

The following theorem, which is \cite[corollaire 2]{DL20}, will be used implicitly in the sequel:

\begin{theorem} \label{thm:DL}
Let $L/H$ be a Galois extension of division rings with finite Galois group. Assume $H$ is of finite dimension over its center $h$ and let $\ell$ denote the center of $L$. Then $\ell$ is a Galois field extension of $h$ and ${\rm{Gal}}(\ell/h) \cong {\rm{Gal}}(L/H)$.
\end{theorem}

A ring $R \not= \{0\}$ with $ab \not=0$ for $a \not=0$ and $b \not=0$ is a {\it{right Ore domain}} if, for all $x \not=0$ and $y \not=0$, there are $r, s \in R$ with $xr = ys \not=0$. For a right Ore domain $R$, there is a division ring $H \supseteq R$ each element of which can be written as $ab^{-1}$ with $a \in R$, $b \in R \setminus \{0\}$ (see, e.g., \cite[Theo\-rem 6.8]{GW04}), and $H$ is unique up to isomorphism (see \cite[Proposition 1.3.4]{Coh95}). 

Let $H$ be a division ring and $\sigma$ an automorphism of $H$. We let $H[T, \sigma]$ be the ring of all polynomials $a_0 + a_1 T + \cdots + a_n T^n$ with $n \geq 0$ and $a_0, \dots, a_n \in H$, whose addition is defined componentwise and multiplication fulfills $Ta = \sigma(a) T$ for every $a \in H$. Note that $H[T, \sigma]$ is commutative if and only if $H$ is a field and $\sigma={\rm{id}}_H$. In the sense of Ore (see \cite{Ore33}), $H[T, \sigma]$ is the polynomial ring $H[T, \sigma, \delta]$ in the variable $T$, where the derivation $\delta$ is 0. The ring $H[T, \sigma]$ fulfills $ab \not=0$ for $a \in H[T, \sigma] \setminus \{0\}$ and $b \in H[T, \sigma] \setminus \{0\}$, as the degree is additive on products. Moreover, $H[T, \sigma]$ is a right Ore domain (see, e.g., \cite[Theorem 2.6 and Corollary 6.7]{GW04}). By $H(T, \sigma)$, we then mean the unique division ring containing $H[T, \sigma]$ and every element of which can be written as $ab^{-1}$ with $a \in H[T, \sigma]$ and $b \in H[T, \sigma] \setminus \{0\}$. If $\sigma={\rm{id}}_H$, we write $H[T]$ and $H(T)$ instead of $H[T, {\rm{id}}_H]$ and $H(T, {\rm{id}}_H)$, respectively. 

\subsubsection{Finite embedding problems} \label{sssec:tiganoco_1.2}

First, let $L/H$ and $F/M$ be Galois extensions of division rings with finite Galois groups, and such that the two inclusions $L \subseteq F$ and $H \subseteq M$ hold. We let ${\rm{res}}^{F/M}_{L/H}$ denote the res\-triction map ${\rm{Gal}}(F/M) \rightarrow {\rm{Gal}}(L/H)$ (i.e., ${\rm{res}}^{F/M}_{L/H}(\sigma)(x)=\sigma(x)$ for $\sigma \in {\rm{Gal}}(F/M)$ and $x \in L$), if it is well-defined. 

Unlike the field case, ${\rm{res}}^{F/M}_{L/H}$ is not always well-defined. The next proposition, which relies on the special cases III) and IV) of \cite[\S3.1]{BDL20f}, gives two situations where it is well-defined:

\begin{proposition} \label{prop:restriction}
Let $H$ be a division ring of finite dimension over its center $h$.

\vspace{0.5mm}

\noindent
{\rm{(1)}} Let $L/H$ and $F/H$ be two Galois extensions of division rings with finite Galois groups and such that $L \subseteq F$. Then ${\rm{res}}^{F/H}_{L/H}$ is well-defined.

\noindent
{\rm{(2)}} Let $\sigma$ be an automorphism of $H$ of finite order $m$, let $L/H$ be a Galois extension of division rings with finite Galois group, and let $\tau$ be an automorphism of $L$ of order $m$ extending $\sigma$. Denote the restriction of $\tau$ to the center $\ell$ of $L$ by $\widetilde{\tau}$, and assume 
$$\langle \widetilde{\tau}, {\rm{Gal}}(\ell/h) \rangle \cong \langle \widetilde{\tau} \rangle \times {\rm{Gal}}(\ell/h) \, \, \footnote{Note that ${\rm{Gal}}(\ell/h)$ is a well-defined (finite) group by Theorem \ref{thm:DL}.}.$$
Then $L(T, \tau) / H(T, \sigma)$ is Galois with finite Galois group and ${\rm{res}}^{L(T,\tau) / H(T, \sigma)}_{L/H}$ is a well-defined isomorphism.
\end{proposition}

Now, a {\it{finite embedding problem}} over a division ring $H$ of finite dimension over its center is an epimorphism $\alpha : G \rightarrow {\rm{Gal}}(L/H)$, with $G$ a finite group and $L/H$ a Galois extension. Say that $\alpha$ {\it{splits}} if there is an embedding $\alpha' : {\rm{Gal}}(L/H) \rightarrow G$ with $\alpha \circ \alpha' = {\rm{id}}_{{\rm{Gal}}(L/H)}$. A {\it{weak solution}} to $\alpha$ is a monomorphism $\beta : {\rm{Gal}}(F/H) \rightarrow G$, where $F/H$ is a Galois extension with $L \subseteq F$, such that $\alpha \circ \beta$ is the map ${\rm{res}}^{F/H}_{L/H}$ (which is well-defined by Proposition \ref{prop:restriction}(1)). If $\beta$ is an isomorphism, we say {\it{solution}} instead of weak solution.

\begin{remark} \label{rk:fep}
Let $L/H$ be a Galois extension of division rings with ${\rm{Gal}}(L/H)$ finite. Then $L$ is a field if and only if $H$ is a field (see \cite[lemme 2.1]{BDL20f}). In particular, the above terminology of finite embedding problems generalizes that from the commutative case recalled in \S\ref{ssec:bas_2}, and there is no possible confusion with it.
\end{remark}

Next, let $H$ be a division ring of finite dimension over its center $h$ and $\sigma$ an automorphism of $H$ of finite order $m$. Let $\widetilde{\sigma}$ be the restriction of $\sigma$ to $h$. Let $\alpha : G \rightarrow {\rm{Gal}}(L/H)$ be a finite embedding problem over $H$ and $\tau$ an automorphism of $L$ of order $m$ extending $\sigma$. Let $\widetilde{\tau}$ denote the restriction of $\tau$ to the center $\ell$ of $L$ and assume $\langle \widetilde{\tau}, {\rm{Gal}}(\ell/h) \rangle \cong \langle \widetilde{\tau} \rangle \times {\rm{Gal}}(\ell/h)$. Then, by Proposition \ref{prop:restriction}(2), $L(T, \tau) / H(T, \sigma)$ is Galois with finite Galois group and ${\rm{res}}^{L(T,\tau) / H(T, \sigma)}_{L/H}$ is a well-defined isomorphism. Hence, 
$$\alpha_{\sigma, \tau} = ({\rm{res}}^{L(T, \tau)/H(T, \sigma)}_{L/H})^{-1} \circ \alpha : G \rightarrow {\rm{Gal}}(L(T, \tau)/H(T, \sigma))$$
is a finite embedding problem over $H(T, \sigma)$. A {\it{$(\sigma, \tau)$-geometric solution}} to $\alpha$ is a solution ${\rm{Gal}}(E/H(T, \sigma)) \rightarrow G$ to $\alpha_{\sigma, \tau}$. If $\tau={\rm{id}}_L$ (and so $\sigma={\rm{id}}_H$), we say {\it{geometric solution}} instead of $({\rm{id}}_H, {\rm{id}}_L)$-geometric solution. As in Remark \ref{rk:fep}, $H(T, \sigma)$ is a field if and only if $E$ is.

Moreover, we let $h^{\langle \widetilde{\sigma} \rangle}$ (resp., $\ell^{\langle \widetilde{\tau} \rangle}$) denote the field which consists of all elements of $h$ (resp., of $\ell$) which are fixed under $\widetilde{\sigma}$ (resp., under $\widetilde{\tau}$). By \cite[lemmes 2.4 \& 3.5]{BDL20f}, $\ell^{\langle \widetilde{\tau} \rangle}$ is a Galois field extension of $h^{\langle \widetilde{\sigma} \rangle}$ and the (usual) restriction map ${\rm{res}}^{\ell/h}_{\ell^{\langle \widetilde{\tau} \rangle} / h^{\langle \widetilde{\sigma} \rangle}}$ is an isomorphism. Furthermore, the special case II) of \cite[\S3.1]{BDL20f} shows that the restriction map ${\rm{res}}^{L/H}_{\ell/h}$ is a well-defined isomorphism. We may then consider this finite embedding problem over $h^{\langle \widetilde{\sigma} \rangle}$:
\begin{equation} \label{overline}
\overline{\alpha}_{\sigma, \tau} = {\rm{res}}^{\ell/h}_{\ell^{\langle \widetilde{\tau} \rangle} / h^{\langle \widetilde{\sigma} \rangle}} \circ {\rm{res}}^{L/H}_{\ell/h} \circ \alpha : G \rightarrow {\rm{Gal}}(\ell^{\langle \widetilde{\tau} \rangle} / h^{\langle \widetilde{\sigma} \rangle}).
\end{equation}
The next lemma, which is \cite[lemme 4.2]{BDL20f}, shows why $\overline{\alpha}_{\sigma, \tau}$ is of interest:

\begin{lemma} \label{lemma:geo}
Let $H$ be a division ring of finite dimension over its center $h$ and $\sigma$ an automorphism of $H$ of finite order $m$. Denote the restriction of $\sigma$ to $h$ by $\widetilde{\sigma}$. Let $\alpha : G \rightarrow {\rm{Gal}}(L/H)$ be a finite embedding problem over $H$ and $\tau$ an automorphism of $L$ of order $m$ extending $\sigma$. Denote the restriction of $\tau$ to the center $\ell$ of $L$ by $\widetilde{\tau}$. Assume $\langle \widetilde{\tau}, {\rm{Gal}}(\ell/h) \rangle \cong \langle \widetilde{\tau} \rangle \times {\rm{Gal}}(\ell/h).$ Then $\alpha$ has a $(\sigma, \tau)$-geometric solution, if $\overline{\alpha}_{\sigma, \tau}$ has a geometric solution ${\rm{Gal}}(E/h^{\langle \widetilde{\sigma} \rangle}(T^m)) \rightarrow G$ such that $E \subseteq \ell^{\langle \widetilde{\tau} \rangle} ((T^m))$.
\end{lemma}

In the trivial case $L=H$ (and so $\tau=\sigma$), Lemma \ref{lemma:geo} reduces to this statement: if $H$ is of finite dimension over its center $h$, then $G$ is a Galois group over $H(T, \sigma)$ if there is a Galois field extension $E/h^{\langle \widetilde{\sigma} \rangle}(T^m)$ of group $G$ with $E \subseteq h^{\langle \widetilde{\sigma} \rangle} ((T^m))$. The latter statement actually holds without the assumption that $H$ is of finite dimension over $h$, as shown in \cite[\S2.1]{Beh21}:

\begin{lemma} \label{lemma:geo_2}
Let $H$ be a division ring of center $h$, let $\sigma$ be an automorphism of $H$ of finite order $m$, and let $\widetilde{\sigma}$ be the restriction of $\sigma$ to $h$. Given a finite group $G$, there is a Galois extension of $H(T, \sigma)$ of group $G$, if there is a Galois field extension $E/h^{\langle \widetilde{\sigma} \rangle}(T^m)$ of group $G$ with $E \subseteq h^{\langle \widetilde{\sigma} \rangle} ((T^m))$.
\end{lemma}

\subsection{Proof of Theorem \ref{thm:intro_4}} \label{ssec:tiganoco_2}

Denote the order of $\sigma$ by $m$ and the restriction of $\sigma$ to the center $h$ of $H$ by $\widetilde{\sigma}$. Let $\mathcal{C}$ be the class of finite groups $G$ which occur as the Galois group of a Galois field extension $E/h^{\langle \widetilde{\sigma} \rangle}(T^m)$ with $E \subseteq h^{\langle \widetilde{\sigma} \rangle}((T^m))$. Then $\mathcal{C}$ is non-empty  and it is clear that, if $G \in \mathcal{C}$ and $H \trianglelefteq G$, then $G/H \in \mathcal{C}$. Moreover, let $G \in \mathcal{C}$ and consider any semidirect product $A \rtimes G$ with $A$ finite abelian. As $G \in \mathcal{C}$, there is a Galois field extension $F/h^{\langle \widetilde{\sigma} \rangle}(T^m)$ of group $G$ with $F \subseteq h^{\langle \widetilde{\sigma} \rangle}((T^m))$. Theorem \ref{thm:intro_1} then yields a Galois field extension $E/h^{\langle \widetilde{\sigma} \rangle}(T^m)$ of group $A \rtimes G$ with $E \subseteq F \cdot h^{\langle \widetilde{\sigma} \rangle}((T^m))$. Since $F \subseteq h^{\langle \widetilde{\sigma} \rangle}((T^m))$, we get $E \subseteq h^{\langle \widetilde{\sigma} \rangle}((T^m))$, which implies that $A \rtimes G$ is in $\mathcal{C}$. Consequently, every finite semiabelian group is in $\mathcal{C}$. As every element of $\mathcal{C}$ is a Galois group over $H(T, \sigma)$ by Lemma \ref{lemma:geo_2}, we are done.

\subsection{Finite embedding problems with abelian kernels over division rings} \label{ssec:tiganoco_3}

Let $H$ be a division ring of finite dimension over its center $h$. As shown in \cite[corollaire 5.4]{BDL20f}, any given finite embedding problem $\alpha$ over $H$ with a weak solution has a geometric solution, if $h$ is ample. Taking $\sigma ={\rm{id}}_H$ (and so $(\tau, \tau') = ({\rm{id}}_L, {\rm{id}}_{L'})$) in the next theorem, which relies on Theorem \ref{thm:intro_1} and the material from \S\ref{ssec:tiganoco_1}, yields that the same holds if we replace the assumption ``$h$ is ample" by the one ``${\rm{ker}}(\alpha)$ is abelian", as alluded to at the end of \S\ref{sssec:intro_2.2}.

\begin{theorem} \label{thm:main_2} 
Let $H$ be a division ring of finite dimension over its center $h$ and $\sigma$ an automorphism of $H$ of finite order $m$. Let $\alpha : G \rightarrow {\rm{Gal}}(L/H)$ be a finite embedding pro\-blem over $H$, and let $\tau$ be an automorphism of $L$ of order $m$ extending $\sigma$ and such that $\langle{\widetilde{\tau}} ,{\rm{Gal}}(\ell/h) \cong \langle \widetilde{\tau} \rangle \times {\rm{Gal}}(\ell/h),$ where $\widetilde{\tau}$ denotes the restriction of $\tau$ to the center $\ell$ of $L$. Assume the following two conditions hold:

\vspace{0.5mm}

\noindent
{\rm{(1)}} ${\rm{ker}}(\alpha)$ is abelian,

\vspace{0.5mm}

\noindent
{\rm{(2)}} there exist a weak solution ${\rm{Gal}}(L'/H) \rightarrow G$ to $\alpha$, and an automorphism $\tau'$ of $L'$ of order $m$ extending $\tau$ and such that, if $\widetilde{\tau'}$ is the restriction of $\tau'$ to the center $\ell'$ of $L'$, then
\begin{equation} \label{eq:bla}
\langle \widetilde{\tau'}, {\rm{Gal}}(\ell'/h) \rangle \cong \langle \widetilde{\tau'} \rangle \times {\rm{Gal}}(\ell'/h).
\end{equation}

\noindent
Then $\alpha$ has a $(\sigma, \tau)$-geometric solution.
\end{theorem}

\begin{proof}
As (2) holds, we may apply the weak$\rightarrow$split reduction for finite embedding problems over division rings (see \cite[proposition 5.3]{BDL20f}) to get the existence of a finite split embedding problem $\alpha' : G' \rightarrow {\rm{Gal}}(L'/H)$ over $H$ such that ${\rm{ker}}(\alpha) \cong {\rm{ker}}(\alpha')$ and such that every $(\sigma, \tau')$-geometric solution to $\alpha'$ yields a $(\sigma, \tau)$-geometric solution to $\alpha$. By the former conclusion and (1), ${\rm{ker}}(\alpha')$ is abelian. Now, as \eqref{eq:bla} holds, we may apply Lemma \ref{lemma:geo}: $\alpha'$ has a $(\sigma, \tau')$-geometric solution, if the finite embedding problem ${\overline{\alpha'}}_{\sigma, \tau'} : G' \rightarrow {\rm{Gal}}(\ell'^{\langle {\widetilde{\tau'}} \rangle} / h^{\langle \widetilde{\sigma} \rangle})$ over $h^{\langle \widetilde{\sigma} \rangle}$ (see \eqref{overline}) has a geometric solution ${\rm{Gal}}(E/h^{\langle \widetilde{\sigma} \rangle}(T^m)) \rightarrow G'$ with $E \subseteq \ell'^{\langle \widetilde{\tau'} \rangle} ((T^m))$. Note that ${\overline{\alpha'}}_{\sigma, \tau'}$ splits (as this holds for $\alpha'$) and ${\rm{ker}}({\overline{\alpha'}}_{\sigma, \tau'})$ is abelian (as ${\rm{ker}}(\alpha')$ is). Hence, such a geometric solution to ${\overline{\alpha'}}_{\sigma, \tau'}$ exists by Theorem \ref{thm:intro_1}. This concludes the proof.
\end{proof}

\section{Concluding remarks}

Given a finite split embedding problem $\alpha : G \rightarrow {\rm{Gal}}(L/k)$ over a field $k$, producing a geometric solution $\beta : {\rm{Gal}}(E/k(T)) \rightarrow G$ with $E \subseteq L((T))$, as Theorem \ref{thm:intro_1} does if ${\rm{ker}}(\alpha)$ is abelian, has several interests, as we have already seen in \S\ref{sec:proof_2} and \S\ref{sec:tiganoco}. Here are two others.

\vspace{2mm}

\noindent
(1) Recall that the {\it{level}} of a field $k$ is equal to either $\infty$ if -1 cannot be written as a sum of finitely many squares in $k$ or the smallest positive integer $n$ such that there exist $x_1, \dots, x_n \in k$ with $-1 = x_1^2 + \cdots + x_n^2$ otherwise. By a well-known result of Pfister (see, e.g., \cite[Chapter XI, Theorem 2.2]{Lam05}), the level of a field is either $\infty$ or a power of 2.

\vspace{1mm}

\noindent
{\it{{\rm{($*$)}} Let $k$ be a number field of level at least 4 and $G$ a finite group which occurs as the Galois group of a Galois field extension $E/k(T)$ with $E \subseteq k((T))$. Then $G$ is the Galois group of a Galois extension of the division ring $H_k$ of quaternions with coefficients in $k$ (i.e., $H_k=k\oplus k\mathbf{i}\oplus k\mathbf{j}\oplus k\mathbf{k}$ with $\mathbf{i}^2= \mathbf{j}^2= \mathbf{k}^2 =\mathbf{i}\mathbf{j}\mathbf{k}=-1$).}}

\vspace{1mm}

\noindent
Indeed, as $k$ has level $\geq$ 4, the Hasse--Minkowski theorem (see, e.g., \cite[p. 170]{Lam05}) yields an equivalence class $v$ of absolute values on $k$ such that the completion $k_v$ of $k$ at $v$ has level $ \geq 4$. As $E_0/k=k/k$, we have $E_0 \subseteq k_v$. Then pick $t_0 \in k$ such that the specialization $E_{t_0}/k$ of $E/k(T)$ at $t_0$ is of group $G$ and such that $E_{t_0} \subseteq k_v$. Hence, $E_{t_0}$ has level $\geq 4$ and it remains to use \cite[th\'eor\`eme 7]{DL20} to conclude the proof of ($*$).

\vspace{2mm}

\noindent
(2) Let $h : \overline{\Qq} \rightarrow \Rr_{\geq 0}$ be the (absolute logarithmic) Weil height. Recall that a field $L \subseteq \overline{\Qq}$ has the {\it{Northcott property}} (Property (N)) if $\{x \in L \, \, : \, \, h(x) < T\}$ is finite for every $T >0$. By Northcott's theorem (see \cite[Theorem 1]{Nor49}), every number field has Property (N). But examples of infinite algebraic extensions of $\Qq$ with Property (N) are more sparse.

In \cite{CF21}, Checcoli and Fehm give many such examples by proving that, given a sequence $(G_n)_{n \geq 1}$ of finite solvable groups, there is an infinite Galois field extension $L$ of $\Qq$ such that ${\rm{Gal}}(L/\Qq) = \prod_{n=1}^\infty G_n$, the completion of $L/\Qq$ at $p$ is a finite extension of $\Qq_p$ for every prime number $p$, and $L$ has Property (N). From the proof of \cite[Theorem 1.3]{CF21} (see also \cite[Remark 2.5]{CF21}), the assumption that $G_n$ is solvable is used only to guarantee the existence of ``many" Galois extensions of $\Qq$ of group $G_n$ which are totally split at any finitely many given prime numbers. Hence, the result of Checcoli and Fehm remains true if each $G_n$ ($n \geq 1$) is either solvable or the Galois group of a Galois extension $E/\Qq(T)$ with $E \subseteq \Qq((T))$.

\bibliography{Biblio2}
\bibliographystyle{alpha}

\end{document}